%% file: main.tex
\documentclass[11pt]{amsart}
\usepackage{latexsym,amsfonts,amssymb,amsthm,amsmath,mathrsfs,color,amscd,graphicx,fullpage,parskip,hyperref,bm,bbm,dsfont}

\usepackage{comment}
\includecomment{uncomment}
\usepackage{commath,mathtools,setspace}
\usepackage{paralist}
\usepackage{extarrows}
\usepackage{enumitem}

\excludecomment{comment} 

\newcommand{\s}[1]{{\mathcal #1}}

\newcommand{\bb}[1]{{\mathbb #1}}

\DeclareMathOperator{\Div}{div}

\newtheorem{theorem}{Theorem} 
\newtheorem{corollary}[theorem]{Corollary}

\newtheorem{lemma}[theorem]{Lemma}
\newtheorem{proposition}[theorem]{Proposition}

\newtheorem{definition}[theorem]{Definition}

\newtheorem{remark}[theorem]{Remark}

\numberwithin{equation}{section}
\numberwithin{theorem}{section}

\newcounter{step}
\begin{document}

	\title[MFGC Fractional Laplacian]
	{Mean Field Games of Controls with Fractional Laplacian}
	
	\author{P. Jameson Graber}
	\thanks{The authors are grateful to be supported by National Science Foundation through NSF Grant DMS-2045027.}
	\address{Baylor University, Department of Mathematics;\\
		Sid Richardson Building\\
		1410 S.~4th Street\\
		Waco, TX 76706
	}
	\email{Jameson\_Graber@baylor.edu}

    \author{Elizabeth Matter}

    \email{ellie\_carr3@baylor.edu}
 
	\author{Jesus Ruiz Bolanos}

    \email{feu25@txstate.edu}
	
	\date{\today}   
	
	\begin{abstract}
		We analyze a fractional mean field game of controls system, showing existence of solutions when the order of the fractional Laplacian is $s\in\del{\frac{1}{2},1}$. Here the running cost depends on the distribution $\mu$ of not only the states but also optimal strategies. The coupling is assumed to satisfy the Lasry-Lions monotonicity condition. We derive three types of a priori estimates on solutions. First, we use the monotonicity condition to derive moment estimates on $\mu$. Second, we derive abstract estimates on fractional parabolic equations and apply them to the mean field game. Third, we derive new estimates on the time regularity of the distribution $\mu$ by analyzing the associated L\'evy process. We apply these estimates and the Leray-Schauder fixed point theorem to establish existence of solutions.
	\end{abstract}
	
	
	\maketitle
	
	
	\section{Introduction}
	The purpose of this article is to analyze a system of partial differential equations having the form
	\begin{equation}\label{eq:MFG}
    \begin{cases}
	(i)& -\partial_tu + (-\Delta)^s u + H(x,D u,\mu) = 0, \quad {\text{in}}\ \bb{T}^d\times(0,T),\\
	(ii)& \partial_t m + (-\Delta)^s m -\Div (D_p H(x,D u,\mu)m)=0, \quad {\text{in}}\ \bb{T}^d\times(0,T),\\
	(iii)& \mu(t) = \del{I_d,-D_p H(x,D u,\mu) }\sharp m(t), \quad \text{in} \ [0,T],\\
     (iv)& m(x,0)=m_0(x),\ \  u(x,T)= u_T(x), \quad {\rm{in}}\ \bb{T}^d.
     \end{cases}
	\end{equation}
    System \eqref{eq:MFG} corresponds to a type of mean field game, or MFG.
    Briefly, this means there exists a continuum of agents whose individually solve an optimal control problem whose value function is given by $u$, while the population density is modeled by $m$.
    We recall that mean field games were introduced in \cite{lasry07,huang2006large,huang2007large}, and the interested reader may refer to several monographs and lecture notes for an overview \cite{carmona2017probabilistic,carmona2017probabilisticII,bensoussan2013mean,gomes2014mean,cardaliaguet2010notes,cardaliaguet2021cetraro}.
    In the case of System \eqref{eq:MFG}, the coupling is through the probability distribution $\mu(t)$, which is the distribution of both states and controls of all agents.
	For that reason, this system is referred to as a mean field game of controls, or MFGC \cite{cardaliaguet2017mfgcontrols,kobeissi2021classical,kobeissi2022mean}, though in previous literature such systems have also been called ``extended mean field games'' \cite{gomes2014extended,gomes2016extended}.
	What distinguishes MFGC from other MFG is seen in Equation \eqref{eq:MFG}$(iii)$: the optimal control problem for each individual depends on the distribution not only of states but also of all optimal strategies.
	All MFG involve a fixed point problem: for every distribution there is an optimal strategy, for every strategy there is a distribution, and there is equilibrium precisely when these coincide.
	For MFGC, the distribution is given not only by $m(t)$ but also by $\mu(t)$ defined in Equation \eqref{eq:MFG}$(iii)$, which can be seen as an additional fixed point problem that must be solved before resolving the larger problem of equilibrium.
    
    Compared to the extensive literature on mean field games systems, there are relatively few general results on mean field games of controls.
    Kobeissi proved the existence and uniqueness of classical solutions to MFGC in the uniformly parabolic setting (that is, with a Laplacian in place of the fractional Laplacian seen in \eqref{eq:MFG} above). See \cite{kobeissi2021classical,kobeissi2022mean}, where the first study makes use of ``smallness'' conditions to prove well-posedness, while the latter makes extensive use of the Lasry-Lions monotonicity condition for mean field games of controls; in this direction see also \cite[Section 5]{cardaliaguet2017mfgcontrols}.
    We also refer to \cite{gomes2014extended,gomes2016extended} for some early results in this direction.
    More recent works have treated the uniformly parabolic case on bounded domains with Dirichlet or Neumann boundary conditions \cite{bongini2024mean,graber2025meanfieldgamescontrols}.
    Mean field games of controls also have natural applications to economics \cite{gueant2011mean,achdou2014partial}.
    Several works, including by the first author, have been written on a mean field games model for Cournot competition, which is a natural example of a mean field game of controls \cite{chan2015bertrand,chan2017fracking,graber2018existence,graber2020commodities,graber2018variational,graber2021nonlocal,graber2022parameter,graber2022master,camilli2025learning}.

    Most of these references deal with second-order parabolic models, modeling stochastic games with Brownian motion, or else first-order models modeling deterministic games (except for \cite{graber2021nonlocal}, on a Cournot model with both Brownian motion and jumps). 
    In this article, we are interested in studying a mean field game of controls with a nonlocal diffusion, which models a game where the state of players can jump; see e.g.~\cite{chan2017fracking}.
    We have chosen the fractional Laplacian as a benchmark case; more general nonlocal diffusion generators would be interesting for applications.
    Mean field games with a fractional Laplacian have been studied by a few authors.
    Cesaroni et al.~proved the existence of classical solutions to a stationary model \cite{cesaroni2017stationary}, while Cirant and Goffi proved a similar result for time-dependent problems in \cite{cirant2019existence}; both results are set on the torus, as in our case.
    Ersland and Jakobsen proved the existence of classical solutions to mean field games with nonlocal diffusion on the whole space \cite{ersland2021fractionalmfg}, and Chowdhury et al.~have proved similar results for fully nonlinear systems \cite{chowdhury2024fully}.
    See also some related results on the master equation \cite{jakobsen2023master}, long-time behavior \cite{ersland2025long}, and numerics \cite{chowdhury2023numerical}.

	As far as we know, this is the first article on a general class of mean field games of controls with a nonlocal diffusion term.
    Our goal is to prove existence of classical solutions to \eqref{eq:MFG}, in the case where $\frac{1}{2} < s < 1$.
    To prove this, we will apply the Leray-Schauder fixed point theorem, which relies mainly on first proving a sequence of a a priori estimates.
    We draw inspiration from the results of Kobeissi on MFGC with monotone couplings \cite{kobeissi2022mean}, but certain regularity estimates are not feasible when only the fractional Laplacian is present.
    To adapt to this setting, we have found many techniques found in \cite{cirant2019existence} very helpful.

    The remainder of this article is organized as follows.
    In Section \ref{sec:prelim}, we introduce various basic notions and abstract results on fractional parabolic equations that will be useful in what follows.
    Section \ref{sec:a priori} is the heart of our study, where we prove a priori estimates on classical solutions of \eqref{eq:MFG}.
    In Section \ref{sec:existence} we apply the Leray-Schauder fixed point theorem to prove existence of solutions.
    Finally, we prove uniqueness employing standard arguments based on the Lasry-Lions monotonicity condition.

	\section{Preliminaries} \label{sec:prelim}
	
	We denote by $\bb{R}^d$ the standard $d$-dimensional Euclidean space.	
	By $\bb{T}^d$ we mean the $d$-dimensional torus $\bb{R}^d/\bb{Z}^d$, which can also be embedded into $\bb{R}^d$ as $\bb{T}^d = [0,1]^d$ with periodic boundary conditions. We define $Q_T:=\bb{T}^d\times[0,T].$

    Throughout this manuscript, when deriving estimates, we will use $C$ as a placeholder for some generic constant not depending on the unknowns, whose precise value can change from line to line.
    \subsection{Spaces of probability measures}		
\input{wasserstein_spaces}


We denote by
$\s{P}_2\left(\bb{R}^d\right)$
the subset of
$\s{P}\left(\bb{R}^d\right)$
of probability measures
with finite second moments,
and
$\s{P}_{\infty}\left(\bb{R}^d\times\bb{R}^d\right)$
the subset of measures $\mu$ in $\s{P} \left(\bb{R}^d\times\bb{R}^d\right)$
with a second marginal compactly supported.
For $\mu\in\mathcal{P}_{\infty}\left(\bb{R}^d\times\bb{R}^d\right)$
and $\tilde{q}\in[1,\infty)$,
we define the quantities
$\Lambda_{\tilde{q}}(\mu)$ and $\Lambda_{\infty}(\mu)$ by,
\begin{equation}
    \label{eq:defL}
\begin{aligned}
    \Lambda_{\tilde{q}}(\mu)
    &=
    \left(\int_{\bb{R}^d\times\bb{R}^d}
    \abs{\alpha}^{\tilde{q}}\dif\mu\left( x,\alpha\right)\right)^{\frac1{\tilde{q}}},
    \\
    \Lambda_{\infty}(\mu)
    &=
    \sup\left\{\abs{\alpha},
    (x,\alpha)\in{\rm supp }\mu\right\}.
\end{aligned}
\end{equation}
For $R>0$, we denote by
$\s{P}_{\infty,R}\left(\bb{R}^d\times\bb{R}^d\right)$
the subset of measures $\mu$ in
$\s{P}_{\infty}\left(\bb{R}^d\times\bb{R}^d\right)$
such that $\Lambda_{\infty}\left(\mu\right)\leq R$.
For $m\in\s{P}\left(\bb{R}^d\right)$,
we call
$\s{P}_m\left(\bb{R}^d\times\bb{R}^d\right)$
the set of all measures $\mu\in\s{P}\left(\bb{R}^d\times\bb{R}^d\right)$
such that there exists
$\alpha^{\mu}\in L^{\infty}\left( m\right)$
satisfying
$\mu = (I_d,\alpha^{\mu})\#m$.
Here, $\Lambda_{\tilde{q}}(\mu)$ and $\Lambda_{\infty}(\mu)$
defined in \eqref{eq:defL} are given by
\begin{equation*}
    \Lambda_{\tilde{q}}(\mu)
    =
    \norm{\alpha^{\mu}}_{{L^{\tilde{q}}(m)}},
    \;\;
    \Lambda_{\infty}(\mu)
    =
    \norm{\alpha^{\mu}}_{{L^{\infty}(m)}}.
\end{equation*}

\subsection{Function spaces}
	For any measurable subset $E$ of $\bb{R}^d$, the space $L^p(E)$ denotes the usual Lebesgue space with norm $(\int_E \abs{f}^p)^{1/p}$ for $1 \leq p < \infty$ and $\operatorname{ess sup} \abs{f}$ for $p = \infty$.
	On any metric space $X$, the space $\s{C}^0(X)$ is the set of all bounded continuous real-valued functions on $X$, which has the norm $\enVert{f}_0 = \sup_{x \in X} \abs{f(x)}$.
	
	For any closed set $E \subset \bb{R}^d$, we will now define the H\"older spaces $\s{C}^\beta(E)$. Let $\beta\in(0,1]$ and $k$ a nonnegative integer. $u:E\to \bb{R}$ is in $\s{C}^{k+\beta}(E)$ if $u\in \s{C}^{k}(E)$ and 
	\begin{equation*}
	    [D^i u]_{\s{C}^\beta} :=\sup_{x\ne y \in E} \dfrac{\abs{D^i u(x)-D^i u(y)}}{\abs{x-y}^\beta}<\infty
	\end{equation*}
	for each multiindex $i$ with $\abs{i}=k$.
	
	We can define for $\alpha,\beta \in (0,1)$ the H\"older space $\s{C}^{\alpha,\beta}(E \times [a,b])$ for any closed set $E \subset \bb{R}^d$ and any $-\infty \leq a < b \leq \infty$ as the space of continuous functions $u$ such that 	
	$$[u]_{\s{C}^{\alpha,\beta}}:=[u]_{\s{C}_x^\alpha}+[u]_{\s{C}_t^\beta}<\infty$$	
	where	
	$$[u]_{\s{C}_x^\alpha}:=\sup_{t\in[a,b]}[u(\cdot,t)]_{\s{C}^\alpha}, \qquad [u]_{\s{C}_t^\beta}:=\sup_{x\in  E}[u(x,\cdot)]_{\s{C}^\beta}.$$
	
	We also use the standard definitions of Sobolev spaces. $W^{k,p}(E)$ are functions in $L^p(E)$ with weak derivatives up to degree $k$ in $L^p(E)$. This space has as norm:
	
	\begin{equation}
		\enVert{u}_{W^{k,p}}=
		\begin{cases}
			\sum_{\abs{\alpha}\leq k} \enVert{ D^\alpha u}_{L^p}, & 1\leq p<\infty\\
			\sum_{\abs{\alpha}\leq k} \enVert{ D^\alpha u}_{L^\infty}, & p=\infty
		\end{cases}
	\end{equation}

	We also use the following \emph{fractional} Sobolev spaces. For $k\geq0$ $W^{k,p}(E)$ are functions in $L^p(E)$ with (distributional) derivatives up to degree $k$ in $L^p(E)$.
	 See, e.g.~\cite{di2012hitchhikerʼs}.

For $\mu\in \bb{R}$ and $p\in(1,\infty)$, the Bessel potential space $H_p^\mu(\bb{T}^d)$ is the space of all distributions $u$ such that $(I-\Delta)^{\frac{\mu}{2}}u \in L^p(\bb{T}^d)$, where:

$$(I-\Delta)^{\frac{\mu}{2}}u(x)=\sum_{k\in \bb{Z}^d} (1+4\pi^2\vert k\vert^2)^{\frac{\mu}{2}}\widehat{u}(k)e^{2\pi ik\cdot x}, \qquad \widehat{u}(k)=\int_{\bb{T}^d} u(x)e^{-2\pi i k\cdot x}dx$$

These spaces are given the norm:

$$\norm{u}_{\mu,p}:= \norm{ (I-\Delta)^{\frac{\mu}{2}} u}_p$$

Let $\mu\in \bb{R}$, $p\in(1,\infty)$. Denote $\mathbb{H}_p^\mu(Q):=L^p(0,T;H_p^\mu(Q))$ the space of measurable functions $u:(0,T)\to H_p^\mu(Q)$ for which the following norm is finite:

$$\lVert u \rVert_{\mathbb{H}_p^\mu(Q)}=\left( \int_0^T  \lVert u(\cdot,t) \rVert_{\mu,p}^p dt \right)^{\frac{1}{p}}. $$

We will also refer to the space $\mathcal{H}^\mu_p(Q)$ as the space of functions $u\in\bb{H}^\mu_p(Q)$ with $\partial_tu\in\bb{H}^{\mu-2s}_p(Q)$ with finite norm
$$\enVert{u}_{\mathcal{H}_p^\mu(Q)}:=\enVert{u}_{\bb{H}_p^\mu(Q)}+\enVert{\partial_tu}_{\bb{H}^{\mu-2s}_p(Q)},$$
as in \cite{chang2012stochastic}.
\subsection{The Fractional Laplacian}
Let $u:\bb{T}^d\to \bb{R}$, $s\in(0,1)$. The fractional Laplacian on the torus is given by

$$(-\Delta_{\bb{T}^d})^s u(x):=(2\pi)^{2s} \sum_{k\in \bb{Z}^d} \vert k\vert^{2s} \widehat{u}(k)e^{2\pi i k\cdot x}$$

For simplicity, it is written $(-\Delta)^s$.

Alternatively we may define it using the integral definition (cf.~\cite{garroni1992green,applebaum2009levy,di2012hitchhikerʼs}).
For $u:\bb{T}^d\to\bb{R}$, i.e.~$u:\bb{R}^d \to \bb{R}$ with $u(x+\kappa) = u(x)$ for all $\kappa \in \bb{Z}^d$, we can define $(-\Delta)^s u$ as follows
    \begin{equation}
        (-\Delta)^su(x)=C(d,s)P.V.\int_{\bb{R}^d}\frac{u(x)-u(y)}{\abs{x-y}^{d+2s}}\dif y,
        \quad C(d,s) := \del{\int_{\bb{R}^d} \frac{1- \cos(y_1)}{\abs{y}^{d+2s}}\dif y}^{-1}.
    \end{equation}
If $u \in \s{C}^2$, then we can rewrite the integral as
    \begin{equation}
        (-\Delta)^su(x)=-\frac{1}{2}C(d,s)P.V.\int_{\bb{R}^d}\frac{u(x+y)+u(x-y)-2u(x)}{\abs{y}^{d+2s}}\dif y,
    \end{equation}
and by taking the Fourier transform we see that $\widehat{(-\Delta)^su}(\kappa) = (2\pi \abs{\kappa})^{2s}\hat u(\kappa)$, so that the two definitions are equivalent.
From the integral definition we see that $(-\Delta)^s$ is the generator of the transition semigroup of a $d$-dimensional, $2s$-stable pure jumps L\'evy process with associated L\'evy measure given by $\nu(\dif z)=C(d,s)\frac{\dif z}{|z|^{d+2s}}.$

Now let $T(t)$ be the semigroup generated by the fractional Laplacian. Then $T(t)$ is a semigroup of contractions on $L^p$ for all $p \in [1,\infty]$, i.e.
    \begin{equation}
        \enVert{T(t) f}_p\leq \enVert{f}_p, \quad \forall p\in[1,\infty].
    \end{equation}
    Note for any $f\in \s{C}^\infty(\bb{T}^d)$ and multi-indices $k,m,$
    \begin{equation}
        \enVert{D^{k+m}T(t)f}_p\leq Ct^{-\frac{k}{2s}}\enVert{D^mf}_p \quad \forall p\in[1,\infty].
    \end{equation}
    These properties generalize to the following lemma which can be found in \cite{cirant2019existence}.
    \begin{lemma}\label{lem: sgbounds}
        \begin{enumerate}
            \item For any $p>1$ and $\nu\in\bb{R},\gamma\geq 0$, we have for all $f\in H_p^\nu(\bb{T}^d)$
            $$\enVert{T(t) f}_{\nu+\gamma,p}\leq Ct^{-\gamma/2s}\enVert{f}_{\nu,p},$$
            where $C=C(\nu,\gamma,d,s,p)$.
            \item For any $\theta\in[0,s]$ and $p>1$, there exists a constant $C=C(d,s,p,\theta)$ such that, for all $f\in H_p^{2\theta}(\bb{T}^d)$, it holds that
            $$\enVert{T(t)f-f}_p\leq Ct^{\theta/s}\enVert{f}_{2\theta,p}.$$
        \end{enumerate}
    \end{lemma}    

    We also will make use of the following embedding Theorem from \cite{cirant2019existence}
    \begin{theorem}\label{thm: embedding}
        Let $\varepsilon>0$, $\mu\in\bb{R}$, $p>1$, $u\in\mathcal{H}^\mu_p(Q_T)$, and $u(0)\in H^{\mu-2s/p+\varepsilon}(\bb{T}^d)$. If $\beta$ is such that 
        $$\frac{s}{p}<\beta<s,$$
        then $u\in \s{C}^{\frac{\beta}{s}-\frac{1}{p}}\del{[0,T];H^{\mu-2\beta}_p(\bb{T}^d)}.$ In particular, there exists $C>0$ depending on $d,p,\beta,T,\varepsilon$ such that
        $$\enVert{u(\cdot,t)-u(\cdot,\tau)}^p_{\mu-2\beta,p}\leq C\abs{t-\tau}^{\frac{\beta}{s}p-1}\del{\enVert{u}_{\mathcal{H}_p^\mu(Q_T)}+\enVert{u(0)}_{\mu-2s/p+\varepsilon,p}}$$
        for $0\leq t,\tau\leq T.$ Hence,
        \begin{equation}
            \enVert{u}_{\s{C}^{\frac{\beta}{s}-\frac{1}{p}}\del{[0,T];H^{\mu-2\beta}_p(\bb{T}^d)}}\leq C\del{\enVert{u}_{\mathcal{H}_p^\mu(Q_T)}+\enVert{u(0)}_{\mu-2s/p+\varepsilon,p}}.
        \end{equation}
    \end{theorem}

\subsection{Preliminary lemmas on fractional parabolic equations}
Later we will make use of some basic results regarding the following fractional parabolic equation:
\begin{equation}\label{eq: fractional parabolic}
    \begin{cases}
        \partial_tu+(-\Delta)^su+f(x,t)=0,\quad &\text{in } Q_T,\\
        u(x,T)=u_T(x),\quad &\text{in }\bb{T}^d.
    \end{cases}
\end{equation}
The following two lemmas will be useful in the ``bootstrapping" portion of our derivation of a priori estimates for solutions to the fractional Hamilton Jacobi Equation appearing in \eqref{eq:MFG}.
\begin{lemma}\label{lem: grad holder}
        Suppose $s>\frac{1}{2}$, $\varepsilon > 0$, $f\in L^\infty(Q_T)$, and $u_T \in \s{C}^{2s + \varepsilon}(\bb{T}^d)$. Let $u$ be the solution of \eqref{eq: fractional parabolic}.
        Then there exist constants $C > 0$ and $\beta \in (0,1)$ independent of $u$ and $f$ such that
        \begin{equation} \label{abstract grad estimate}
            \enVert{Du}_{\s{C}^{\beta,\frac{\beta}{2s}}(Q_T)}\leq C\del{\enVert{f}_\infty + \enVert{u_T}_{\s{C}^{2s + \varepsilon}}}
        \end{equation}
\end{lemma}
\begin{proof}
        Observe that $u_T \in H_p^{2s + \varepsilon}(\bb{T}^d)$ and $f \in L^p(Q_T)$ for all $p > 1$.
        It follows from \cite[Theorem B.3]{cirant2019existence} that there exists a unique solution $u \in \bb{H}_p^{2s}$ for arbitrarily large $p$.
        It is enough to prove estimate \eqref{abstract grad estimate}.
        We will do this for $u_T = 0$ and $f$ arbitrary; a similar argument can be used for $f = 0$ and $u_T$ arbitrary, and the general case follows by linearity.
        
        Set $A=(-\Delta)^s$ and let $T(t)$ be  the semigroup generated by $A$. Now for $\gamma\in (0,1)$, we have $A^\gamma=(-\Delta)^{s\gamma}$, and for all $p > 1$ we have
        $$\enVert{A^\gamma T(t)f}_p\leq \frac{C}{t^\gamma}\enVert{f}_p$$
        as a consequence of Lemma \ref{lem: sgbounds}.
        Using this, 
        \begin{equation} \label{Agu leq fp}
        \begin{split}
            \enVert{A^\gamma u(t)}_p&=\enVert{\int_0^tA^\gamma T(t-s)f(s)\dif s}_p\\
            &\leq\int_0^t\enVert{A^\gamma T(t-s)f(s)}_p\dif s\\
            &\leq C\int_0^t(t-s)^{-\gamma}\enVert{f(s)}_p\dif s\\
            &\leq C t^{1-\gamma}\enVert{f}_p.
        \end{split}
        \end{equation}
        Thus for $p>\frac{d}{2s\gamma}$,  
        \begin{equation*}
            \begin{aligned}
                \enVert{u(t)}_{\s{C}^{2s\gamma-d/p}(\bb{T}^d)}
                &\leq C\enVert{u(t)}_{H_p^{2s\gamma}(\bb{T}^d)} \quad \text{(by \cite[Lemma 2.5]{cirant2019existence})}\\
                &\leq C\del{\enVert{u(t)}_{p} + \enVert{A^\gamma u(t)}_{p}} \quad \text{(see \cite[Remark 2.3]{cirant2019existence})}\\
                &\leq C\enVert{f}_{p} \quad \text{(by \eqref{Agu leq fp})}\\
                &\leq C\enVert{f}_{\infty}
            \end{aligned}
        \end{equation*}
        where $C$ is independent of $u$ and $f$.
        Therefore if we take $p$ large and $\gamma$ close enough to $1$ so that $2s\gamma-d/p-1>0$, we get $\enVert{Du(t)}_{\s{C}^\beta(\bb{T}^d)}\leq C\enVert{f}_\infty$ for $2s\gamma-d/p-1>\beta>0.$
        Now for the time regularity of the gradient of $u$, let $t,\tau\in[0,T]$ and without loss of generality suppose $t>\tau$. Then
        \begin{equation}
            \begin{split}
                \enVert{A^\gamma u(t)-A^\gamma u(\tau)}_p&=\enVert{\int_0^tA^\gamma T(t-s)f(s)\dif s-\int_0^\tau A^\gamma T(\tau-s)f(s)\dif s}_p\\
                &\leq \int_0^\tau \enVert{A^\gamma \del{T(t-s)-T(\tau-s)}f(s)}_p\dif s+\int_\tau^t\enVert{A^\gamma T(t-s)f(s)}_p\dif s\\
                &\leq \int_0^\tau \int_{\tau-s}^{t-s}\enVert{A^{1+\gamma} e^{A\rho}f(s)}_p\dif\rho\dif s+\int_\tau^t\enVert{A^\gamma T(t-s)f(s)}_p\dif s\\
                &\leq C\enVert{f}_p\del{\int_0^\tau \int_{\tau-s}^{t-s}\rho^{-(1+\gamma)}\dif\rho\dif s+\int_\tau^t (t-s)^{-\gamma}\dif s}\\
                &\leq C\enVert{f}_p(t-\tau)^{1-\gamma}.
            \end{split}
        \end{equation}
        Now $\gamma<1$ so we can take $\beta<\min\cbr{2s(1-\gamma),2s\gamma-d/p-1}$ so that $\enVert{Du}_{\s{C}^{\beta,\frac{\beta}{2s}}(Q_T)}\leq C\enVert{f}_\infty$.
\end{proof}
The following Lemma is similar to Theorem B.1 in \cite{cirant2019existence}. However, the time regularity depends on $f$ being $\frac{\beta}{2s}$-H\"older continuous in time, while our results will only require $\frac{\beta}{2}$ time regularity. Here we prove the details for such a case, which will be useful in Section \ref{sub: HJ est}.
\begin{lemma}\label{lem: full holder}
    Let $\beta\in(0,1)$ and $s>\frac{1}{2}$ be such that $2s+\beta$ is not an integer. Suppose $u$ is a solution to \eqref{eq: fractional parabolic}. Then
    $$\enVert{\partial_tu}_{\s{C}^{\beta,\frac{\beta}{2}}(Q_T)}+\enVert{(-\Delta)^su}_{\s{C}^{\beta,\frac{\beta}{2}}(Q_T)}\leq C\del{\enVert{u_T}_{\s{C}^{2s+\beta}(\bb{T}^d)}+\enVert{f}_{\s{C}^{\beta,\frac{\beta}{2}}(Q_T)}}$$
\end{lemma}

\begin{proof}
    We may directly apply Theorem B.1 from \cite{cirant2019existence} to show the spatial regularity:
    $$\sup_{t\in[0,T]}\del{\enVert{\partial_tu(\cdot,t)}_{\s{C}^{\beta}(\bb{T}^d)}+\enVert{(-\Delta)^su(\cdot,t)}_{\s{C}^{\beta}(\bb{T}^d)}}\leq C\del{\enVert{u_T}_{\s{C}^{2s+\beta}(\bb{T}^d)}+\sup_{t\in[0,T]}\enVert{f(\cdot,t)}_{\s{C}^{\beta}(\bb{T}^d)}}.$$ 
    Now to show the remaining time regularity we must show that $u_t$ and $(-\Delta)^s$ are $\beta/2$-H\"older continuous with respect to time. Let $A=(-\Delta)^s$ and $T$ be the semigroup generated by $A$. We will use the following estimates from Lemma \ref{lem: sgbounds}
     \begin{align}\label{eq: T estimates}
       \enVert{T(t)}_{L(\s{C}^0)} &\leq M_0\\
       \enVert{AT(t)}_{L(\s{C}^0)} &\leq \frac{M_1}{t}\\
       \enVert{A^2T(t)}_{L(\s{C}^0)} &\leq \frac{M_2}{t^2}\\
       \enVert{AT(t)}_{L(\s{C}^\beta, \s{C}^0)}&\leq \frac{M_{1,\beta}}{t^{1-\frac{\beta}{2s}}}
   \end{align}
   Let $u=u_1+u_2$ where
   \begin{equation}
       \begin{cases}
           u_1(\cdot,t)=\int_0^tT(t-\sigma)\del{f(\cdot,\sigma)-f(\cdot,t)}\dif \sigma,\quad 0\leq t\leq T,\\
           u_2(\cdot,t)=T(t)u_T+\int_0^tT(t-\sigma)f(\cdot,t)\dif \sigma,\quad 0\leq t\leq T.
       \end{cases}
   \end{equation}
   Let $0\leq r\leq t\leq T$, then 
   \begin{equation}
       \begin{split}
           Au_1(\cdot,t)-Au_1(\cdot,r)=\int_0^rA(T(t-\sigma)-T(r-\sigma))(f(\cdot,\sigma)-f(\cdot,r))\dif\sigma\\
           +(T(t)-T(t-r))(f(\cdot,r)-f(\cdot,t))+\int_r^tAT(t-\sigma)(f(\cdot,\sigma)-f(\cdot,t))\dif\sigma.
       \end{split}
   \end{equation}
   Note
   \begin{equation}
      A (T(t-\sigma)-T(r-\sigma))=\int_{r-\sigma}^{t-\sigma}A^2T(\tau)\dif\tau.
   \end{equation}
   Using this, as well as the estimates above yields
   \begin{equation}
       \begin{split}
           \enVert{Au_1(\cdot,t)-Au_1(\cdot,r)}&\leq M_2\int_0^r(r-\sigma)^{\beta/2}\int_{r-\sigma}^{t-\sigma}\tau^{-2}\dif\tau\dif\sigma\sup_{x\in\bb{T}^d}[f(x,\cdot)]_{\s{C}^{\beta/2}}\\
           &+2M_0(t-r)^{\beta/2}\sup_{x\in\bb{T}^d}[f(x,\cdot)]_{\s{C}^{\beta/2}}+M_1\int_r^t(t-\sigma)^{\beta/2-1}\dif\sigma\sup_{x\in\bb{T}^d}[f(x,\cdot)]_{\s{C}^{\beta/2}}\\
           &\leq C\sup_{x\in\bb{T}^d}[f(x,\cdot)]_{\s{C}^{\beta/2}}(t-r)^{\beta/2}
       \end{split}
   \end{equation}
   So $\sbr{Au_1}_{\s{C}_t^\frac{\beta}{2}([0,T])}\leq C\sbr{f}_{\s{C}_t^\frac{\beta}{2}([0,T])}$. Then,
   \begin{equation}
       \begin{split}
           \enVert{Au_2(\cdot,t)-Au_2(\cdot,r)}_{L(\s{C}^0)}  &\leq \enVert{(T(t)-T(r))(Au_T(\cdot)+f(\cdot,0))}_{L(\s{C}^0)} \\
           &\hspace{-.9in}+\enVert{(T(t)-T(r))(f(\cdot,r)-f(\cdot,0))}_{L(\s{C}^0)}+\enVert{(T(t)-1)(f(\cdot,t)-f(\cdot,r))}_{L(\s{C}^0)} \\
           &\hspace{-1in}\leq \int_r^t\enVert{AT(\sigma)}_{L(\s{C}^\beta,\s{C}^0)}\dif\sigma\enVert{Au_0(\cdot)+f(\cdot,0)}_{\s{C}^\beta(\bb{T}^d)}\\
           &\hspace{-.9in}+r^{\frac{\beta}{2}}\enVert{A\int_r^tT(\sigma)\dif\sigma}\sup_{x\in\bb{T}^d}[f(x,\cdot)]_{\s{C}^\frac{\beta}{2}([0,T])}+(M_0+1)(t-r)^{\frac{\beta}{2}}\sup_{x\in\bb{T}^d}[f(x,\cdot)]_{\s{C}^\frac{\beta}{2}([0,T])}\\
           &\hspace{-1in}\leq \frac{2M_{1,\beta}}{\beta}\enVert{Au_T(\cdot)+f(\cdot,0)}_{\s{C}^\beta(\bb{T}^d)}(t-r)^{\frac{\beta}{2s}}\\
           &\hspace{-.9in}+\del{\frac{2M_1}{\beta}+M_0+1}(t-r)^{\frac{\beta}{2}}\sup_{x\in\bb{T}^d}[f(\cdot,x)]_{\s{C}^\frac{\beta}{2}([0,T])}.
       \end{split}
   \end{equation}
   Thus $\sbr{Au_2}_{C_t^\frac{\beta}{2}([0,T])}\leq C\del{\enVert{Au_T}_{\s{C}^{\beta}(\bb{T}^d)}+\enVert{f}_{\s{C}^{\beta,\frac{\beta}{2}}(Q_T)}}$. Therefore $$\enVert{Au}_{\s{C}^{\beta,\frac{\beta}{2}}(Q_T)}\leq C\del{\enVert{u_T}_{\s{C}^{2s+\beta}(\bb{T}^d)}+\enVert{f}_{\s{C}^{\beta,\frac{\beta}{2}}(Q_T)}}.$$ Additionally, since $\enVert{\partial_tu}_{\s{C}^{\beta,\frac{\beta}{2}}(Q_T)}\leq \enVert{Au+f}_{\s{C}^{\beta,\frac{\beta}{2}}(Q_T)}$, the result holds.
\end{proof}

\subsection{Assumptions}

The constants in the assumptions are $C_0$ a positive constant, $q\in(1,\infty)$ an exponent, $\tilde{q}=\frac{q}{q-1}$ its conjugate exponent, and $\beta_0\in(0,1)$ a H\"older exponent.

   \begin{enumerate}[label=(L\arabic*)]
    	\item
    	  \label{hypo:Lstrconv}
    	  $D^2_{\alpha\alpha}L(x,\alpha,\mu)>0$
    	\item
    	\label{hypo:Lreg}
    	$L:\bb{R}^d\times\bb{R}^d\times\s{P}\left(\bb{R}^d\times\bb{R}^d\right)
    	\to\bb{R}$ 
    	is differentiable with respect to $\left( x,\alpha\right)$ and twice differentiable with respect to $x$;
    	$L$ and its derivatives 
    	are continuous on
    	$\bb{R}^d\times\bb{R}^d\times\s{P}_{\infty,R}\left(\bb{R}^d\times\bb{R}^d\right)$
    	for any $R>0$;
    	we recall that $\s{P}_{\infty,R}\left(\bb{R}^d\times\bb{R}^d\right)$
    	is endowed with the weak* topology
    	on measures;
    	we use the notation $D_xL$, $D_\alpha L$
    	and $D_{(x,\alpha)}L$ for
    	respectively the first-order derivatives
    	of $L$ with respect to $x$, $\alpha$
    	and $(x,\alpha)$, and $D^2_{xx}L$ for the second-order derivative of $L$ with respect to $x$.
    	
    	\item
    	\label{hypo:LMono}
    	$L$ satisfies the Lasry-Lions monotonicity condition \cite{cardaliaguet2017mfgcontrols,kobeissi2022mean}:
    	\begin{equation*}
    		\int_{\bb{R}^d\times\bb{R}^d}
    		\left( L\left(  x,\alpha,\mu^1 \right)
    		-L\left(  x,\alpha,\mu^2\right) \right)
    		d\left( \mu^1-\mu^2 \right) (x,\alpha)
    		\geq
    		0.
    	\end{equation*}
    	for any
    	$\mu^1, \mu^2 \in
    	\s{P}\left(\bb{R}^d\times \bb{R}^d \right)$.
    	\item
    	\label{hypo:Lcoer}
    	$L(x,\alpha,\mu)
    	\geq
    	C_0^{-1}|\alpha|^{\tilde{q}}
    	-C_0\left( 1+\Lambda_{\tilde{q}}\left(\mu\right)^{\tilde{q}}\right)$,
    	where $\Lambda_{\tilde{q}}$ is defined in~\eqref{eq:defL},
    	\item
    	\label{hypo:Lbound}
    	$\abs{ L(x,\alpha,\mu)}
    	+\abs{D_xL(x,\alpha,\mu)}
    	\leq
    	C_0\left(1+|\alpha|^{\tilde{q}}
    	+\Lambda_{\tilde{q}}\left(\mu\right)^{\tilde{q}}\right)$.
    \end{enumerate}
    From the Lagrangian $L$, we define the Hamiltonian $H$ by

\begin{equation}\label{hamiltonian}
	H(x,p,\mu)=\sup_{\alpha\in \bb{R}^d}\cbr{ -p \cdot \alpha- L(x,\alpha,\mu)}
\end{equation}
    Additionally assume 
    \begin{enumerate}[label=(H\arabic*)]
        \item \label{DpmuH}
        for each $R > 0$, there exists finite $C_R>0$ and $C'_R\in(0,1)$ such that $\abs{H(x,p,\mu)-H(x,p,\nu)}\leq C_RW_r(\mu,\nu)$ and $\abs{D_{p}H(x,p,\mu)-D_pH(x,p,\nu)}\leq C'_R W_r(\mu,\nu)$ for all $x \in \bb{T}^d$, $p \in B_R$, and $\mu,\nu \in \s{P}_\infty\del{\bb{T}^d \times \bb{R}^d}$ for some $r\geq 1$, 
        \item \label{hypo: Hxbound} $\abs{D_x H(x,p,\mu)}\leq C_0\del{1+\vert p\vert+\Lambda_{\tilde q}(\mu)^{\tilde q}}$,
        \item \label{DxxH}
        for each $R > 0$, there exists a finite $C_R>0$ such that $\abs{D_{xx}^2H(x,p,\mu)}\leq C_R$ and\\ $\abs{D^2_{px}H(x,p,\mu)}\leq C_R$ for all $x \in \bb{T}^d$, $p \in B_R$, and $\mu \in B_R$.    
    \end{enumerate}
    for every $x\in\bb{T}^d$.

    Let us remark that Hypothesis \ref{hypo: Hxbound} is coherent with Equation (18) in \cite{droniou2006fractal}, which implies that if $D_x H(x,p,\mu)$ has a power law growth rate with respect to $p$, then the power cannot be greater than 1.
    Cf.~Assumption (A4) in \cite{ersland2021fractionalmfg}.

    Finally assume $m_0 \in H^{s-1+\varepsilon_0}_2(\bb{T}^d)$ and $u_T \in \s{C}^{2 + \beta_0}$, with $C_0$ a constant such that
   \begin{enumerate}[label=(I)]
    	\item
    	\label{hypo:Monogregx}
    	$\int_{\bb{R}^d}|x|^2dm_0(x)+
    	\norm{m_0}_{s-1+\varepsilon_0,2}
    	+\norm{u_T(\cdot)}_{\s{C}^{2+\beta_0}}
        \leq C_0$.   
   \end{enumerate}

\begin{lemma}
	\label{lem:Hbound}
	Under assumptions
	\ref{hypo:Lstrconv},
    \ref{hypo:Lreg},
	\ref{hypo:Lcoer} and
	\ref{hypo:Lbound},
	the map $H$, defined in \eqref{hamiltonian},
	is differentiable with respect to $x$ and $p$,
	$H$ and its derivatives 
	are continuous on
	$\bb{T}^d\times\bb{R}^d\times{\mathcal P}_{\infty,R}\left(\bb{T}^d\times\bb{R}^d\right)$
	for any $R>0$.
	Moreover there exists ${\widetilde C}_0>0$ a constant which only depends
	on $C_0$ and $q$ such that
	\begin{align}
		\label{eq:Hpbound}
		\abs{ D_pH\left( x,p,\mu\right)}
		&\leq
		{\widetilde C}_0\left( 1+|p|^{q-1}
		+\Lambda_{\tilde{q}}\left(\mu\right)\right),
		\\
		\label{eq:Hbound}
		\abs{ H\left(x,p,\mu\right)}
		&\leq
		{\widetilde C}_0\left(1+|p|^{q}
		+\Lambda_{\tilde{q}}\left(\mu\right)^{\tilde{q}}\right),
		\\
		\label{eq:Hcoer}
		p\cdot D_pH\left(x,p,\mu\right)
		-H\left(x,p,\mu\right)
		&\geq
		{\widetilde C}_0^{-1}|p|^q
		-{\widetilde C}_0\left(1
		+\Lambda_{\tilde{q}}\left(\mu\right)^{\tilde{q}}\right)
	\end{align}
	for any $x\in\bb{T}^d$,
	$p\in\bb{R}^d$ and $\mu\in{\mathcal P}\left(\bb{T}^d\times\bb{R}^d\right)$.
\end{lemma}
Up to replacing $C_0$ with $\max(C_0,{\widetilde C}_0)$,
we can assume that the inequalities in Lemma
\ref{lem:Hbound} are satisfied with $C_0$
instead of ${\widetilde C}_0$.

\subsection{Example}
We give an example of a Hamiltonian satisfying the assumptions from the previous section, which is also suggested by \cite{gomes2016extended}.
Consider 
$$L(x,\alpha,\mu)=\frac{|\alpha+\beta\int \gamma\dif\mu(y,\gamma)|^2}{2}+V(x,\mu)$$ 
and associated Hamiltonian,
$$H(x,p,\mu)=\frac{|p|}{2}^2+\beta p\int\alpha\dif\mu(y,\alpha)-V(x,\mu),$$
where $V$ is bounded, $V(x,\mu)$ is twice differentiable with respect to $x$, $|D_xV|$ and $|D^2_{xx}V|$ are uniformly bounded, and $V$ is Lipschitz and monotone with respect to $\mu$. Let $\beta\in(0,1)$. Claim: \ref{hypo:Lstrconv}-\ref{hypo:Lbound} and \ref{DpmuH}-\ref{DxxH} are satisfied. Indeed,
\begin{itemize}
    \item \ref{hypo:Lstrconv} $$D_{\alpha\alpha}^2L(x,\alpha,\mu)=1>0$$
    \item \ref{hypo:Lreg} $$D_\alpha L(x,\alpha,\mu)=\alpha+\beta\int \gamma\dif\mu(y,\gamma)$$
    $$D_xL(x,\alpha,\mu)=D_xV(x,\mu)$$
    $$D^2_{xx}L(x,\alpha,\mu)=D^2_{xx}V(x,\mu)$$
    \item \ref{hypo:LMono}
    \begin{equation*}
    \begin{split}
    \int_{\bb{R}^d\times\bb{R}^d}&\del{L(x,\alpha,\mu_1)-L(x,\alpha,\mu_2)}\dif(\mu_1-\mu_2)(x,\alpha)\\
    =\beta&\del{\int_{\bb{R}^d\times\bb{R}^d}\alpha\dif(\mu_1-\mu_2)(x,\alpha)}^2+\int_{\bb{R}^d\times\bb{R}^d}\del{V(x,\mu_1)-V(x,\mu_2)}\dif (\mu_1-\mu_2)(x,\alpha)\geq 0   
    \end{split}
    \end{equation*}
    \item \ref{hypo:Lcoer} $$L(x,\alpha,\mu)\geq \frac{(1-\beta^2)|\alpha|^2}{2}-\frac{(1-\beta^2)|\int\gamma\dif\mu(y,\gamma)|^2}{2}+V(x,\mu)\geq C_0^{-1}|\alpha|^2-C_0\del{1+\Lambda_2(\mu)^2}$$
    where $C_0$ depends on the bound on $V$ and $\beta$.
    \item \ref{hypo:Lbound} $$|L(x,\alpha,\mu)|+|D_xL(x,\alpha,\mu)|\leq |\alpha|^2+\abs{\beta\int\gamma\dif\mu(y,\gamma)}^2+V(x,\mu)+D_xV(x,\mu)\leq C_0\del{1+|\alpha|^2+\Lambda_2(\mu)^2}.$$
    \item \ref{DpmuH} 
    $$|H(x,p,\mu)-H(x,p,\nu)|=\abs{\beta p\int \gamma\dif(\mu-\nu)(y,\gamma)+V(x,\mu)-V(x,\nu)}\leq C_RW_q(\mu,\nu).$$ 
    Also,
    $$|D_pH(x,p,\mu)-D_pH(x,p,\nu)|=\abs{\beta\int\gamma\dif(\mu-\nu)(y,\gamma)}\leq \beta W_q(\mu,\nu),$$
    with $C_R<\infty$ for $p\in B_R$.
    \item \ref{hypo: Hxbound} 
    $$|D_xH(x,p,\mu)|=|D_xV(x,\mu)|\leq C,$$
    by our assumptions on $V$.
    \item \ref{DxxH}
    $$|D^2_{xx}H(x,p,\mu)|=|D^2_{xx}V(x,\mu)|\leq C$$
    and 
    $$|D^2_{px}H(x,p,\mu)|=0.$$
\end{itemize}

Now for $H(x,p,\mu)=\frac{\abs{\beta\int\alpha\dif\mu(x,\alpha)+p}^2}{2}+V(x,\mu)$ the assumptions also hold.

\begin{comment}
Consider $H(x,p,\mu)=c(x,\mu)\left(1+|p|^2\right)^{\frac{1}{2}}$, where $c(x,\mu)\geq c_0 >0$. Then, 

\[D_p H(x,p,\mu)=c(x,\mu)\dfrac{p}{\left(1+|p|^2\right)^{\frac{1}{2}}},\]

\[D_{pp}^2 H(x,p,\mu)=\dfrac{\left(1+|p|^2\right)^{\frac{1}{2}}I - \dfrac{pp^T}{ \left(1+|p|^2\right)^{\frac{1}{2}} }  }{\left(1+|p|^2\right)}=\dfrac{\left(1+|p|^2\right)I - pp^T      }{\left(1+|p|^2\right)^{\frac{3}{2}}}  \] 

Notice that the denominator above is always positive.

Let $u\in \mathbb{R}^d$. Then,

\begin{align*}
	u^T\left(\left(1+|p|^2\right)I - pp^T\right)u&=\left(1+|p|^2\right)|u|^2 -(p\cdot u)^2\\
	&\geq |u|^2 + |p|^2|u|^2-(p\cdot u)^2\\
	&\geq |u|^2 >0
\end{align*}

Hence, $H$ is uniformly convex.
\end{comment}

\subsection{Definition of solutions}

\begin{definition}
        We say that $(u,m,\mu)$ is a solution to \eqref{eq:MFG} if
        \begin{itemize}
            \item $u\in \s{C}^{2s,1}(Q_T;\bb{R})$ is a classical solution to \eqref{eq:MFG}(i) and satisfies the terminal condition,
            \item $m\in \s{C}^0\del{[0,T];\s P(\bb{T}^d)}$ is a solution to \eqref{eq:MFG}(ii) in the sense of distributions, and satisfies the initial condition,
            \item $\mu(t)\in\mathcal{P}\del{\bb{T}^d\times\bb{R}^d}$ satisfies \eqref{eq:MFG}(iii) for every $t\in[0,T]$.
        \end{itemize}
    \end{definition}

\subsection{Indexing by $\theta$}
In order to use the Leray-Schauder Theorem, we will consider the following family of Lagrangians indexed by $\theta\in\intoc{0,1}$,
\begin{equation}
	L^{\theta}\left(x,\alpha,\mu\right) 
	=
	\theta
	L\left( x,\theta^{-1}\alpha,\Theta(\mu)\right) ,
\end{equation}
where the map $\Theta:\mathcal{P}\left(\bb{T}^d\times\bb{R}^d\right) 
\rightarrow\mathcal{P} \left(\bb{T}^d\times\bb{R}^d\right)$
is defined by
$\Theta(\mu)=\left( I_d\times\theta^{-1}I_d\right) \#\mu$.

Then the Hamiltonian defined as the Legendre
transform of $L^{\theta}$ is given by
\begin{equation}
	H^{\theta}\left( x,p,\mu\right)
	=
	\theta
	H\left( x,p,\Theta(\mu)\right).
\end{equation}
The definition of the Hamiltonian can naturally be extended
to $\theta=0$ by $H^0=0$,
the associated Lagrangian is $L^0=0$ if $\alpha=0$
and $L^0=\infty$ otherwise.
Consider the following system of MFGC,

\begin{equation}\label{eq:MFGCtheta}
	\left\{
	\begin{array}{lll}
		(i) & -\partial_tu + (-\Delta)^s u + H^\theta(x,D u,\mu) = 0, & {\text{in}}\ \bb{T}^d\times (0,T),\\[5pt]
		(ii) & \partial_t m + (-\Delta)^s m -\Div(D_p H^\theta(x,D u,\mu)m)=0, & {\text{in}}\ \bb{T}^d\times(0,T),\\[5pt]
		(iii) & \mu(t) =\del{I_d, -D_p H^\theta(x,Du,\mu)} \sharp m(t), & \text{in} \ [0,T],\\[5pt]
		(iv) & m(x,0)=m_0(x),\ \  u(x,T)=\theta u_T(x), & {\rm{in}}\ \bb{T}^d,
	\end{array}
	\right.
\end{equation}
for $(x,t)\in\bb{T}^d\times[0,T]$.

Assumptions \ref{hypo:Lstrconv}-\ref{hypo:LMono} are preserved, and inequalities convert to:

\begin{itemize}
		\item
	\label{hypo:Lcoertheta}
	$L^\theta(x,\alpha,\mu)
	\geq
	C_0^{-1} \theta^{1-\tilde{q}}|\alpha|^{\tilde{q}}
	-C_0\theta -C_0\theta^{1-\tilde{q}}\Lambda_{\tilde{q}}(\mu)^{\tilde{q}}  $,
	\item
	\label{hypo:Lboundtheta}
	$\abs{ L^\theta(x,\alpha,\mu)}+\abs{ D_xL^\theta(x,\alpha,\mu)} \leq C\theta +C_0 \theta^{1-\tilde{q}}\left( \vert \alpha\vert^{\tilde{q}} + \Lambda_{\tilde{q}}(\mu)^{\tilde{q}}   \right).	$ 
        \item \label{DpmuHtheta}
         $\abs{H^\theta(x,p,\mu)-H^\theta(x,p,\nu)}\leq C_R\theta W_r(\mu,\nu)$ and $\abs{D_{p}H^\theta(x,p,\mu)-D_pH^\theta(x,p,\nu)}\leq C'_R\theta W_r(\mu,\nu)$ for all $x \in \bb{T}^d$, $p \in B_R$, and $\mu,\nu \in \s{P}_\infty\del{\bb{T}^d \times \bb{R}^d}$ for some $r\geq 1$, $C_R>0$, and $C'_R\in(0,1)$ for all $p\in B_R$,
        \item \label{hypo: Hxboundtheta} $\abs{D_x H^\theta(x,p,\mu)}\leq C\theta\del{1+\vert p\vert}+C\theta^{1-\tilde{q}}\Lambda_{\tilde q}(\mu)^{\tilde q}$,
        \item \label{DxxHtheta}
        $\abs{D^2_{xx}H^\theta(x,p,\mu)}\leq C_H\theta$,
        $\abs{D^2_{px}H^\theta(x,p,\mu)}\leq C_H\theta$, for all $p,\mu$ in a bounded set.
\end{itemize}

The conclusion of Lemma \eqref{lem:Hbound} becomes,
	\begin{align}
		\label{eq:Hpboundtheta}
		\abs{ D_pH^\theta\left( x,p,\mu\right)}
		&\leq
		C_0\theta\left( 1+|p|^{q-1}
		\right)+C_0\Lambda_{\tilde{q}}\left(\mu\right),
		\\
		\label{eq:Hboundtheta}
		\abs{ H^\theta\left(x,p,\mu\right)}
		&\leq
		C_0\theta\left(1+|p|^{q}
		\right)+C_0 \theta^{1-\tilde{q}}\Lambda_{\tilde{q}}\left(\mu\right)^{\tilde{q}},
		\\
		\label{eq:Hcoertheta}
		p\cdot D_pH^\theta\left(x,p,\mu\right)
		-H^\theta\left(x,p,\mu\right)
		&\geq
		C_0^{-1}\theta|p|^q
		-C_0\theta -C_0\theta^{1-\tilde{q}}\Lambda_{\tilde{q}}\left(\mu\right)^{\tilde{q}},
	\end{align}
	for any $(t,x)\in[0,T]\times\bb{T}^d$,
	$p\in\bb{R}^d$ and $\mu\in{\mathcal P}\left(\bb{T}^d\times\bb{R}^d\right)$.

	\section{A priori estimates} \label{sec:a priori}
    We give a general outline of the a priori estimates found in this section:
    \begin{enumerate}
        \item We show in Section \ref{sub: FP est} that if $m$ is a weak solution of the fractional Fokker-Planck equation in the periodic setting with initial condition $m_0$, then $\enVert{m}_{\bb{H}^s_2(Q_T)}\leq C$, where $C$ is a constant that depends on $L^\infty$ bounds on $m_0,$ $H$, and $[\Div H]^-$.
        \item In Section \ref{sub: mu est} we use a result from \cite{kobeissi2022mean} to get uniqueness and moment estimates of the fixed point $\mu$ and we prove bounds on the integral of $\mu$. Both results are needed in proving estimates in \ref{sub: HJ est}.
        \item Section \ref{sub: HJ est} is devoted to the bootstrap argument to get optimal regularity of solutions to System \eqref{eq:MFGCtheta}. This is shown in four steps:
        \begin{enumerate}
            \item In Lemmas \ref{lem: u bdd} and \ref{lem: grad u bound} we show $u$ and its gradient are each bounded by $C\theta$, where $C$ depends on $T$ and constants from the assumptions.
            \item Next we get $L^\infty$ bounds on $H^\theta$ and $D_pH^\theta$. Note that this will also be needed for our result in \ref{sub: FP est} to hold.
            \item We use the result from Lemma \ref{lem: grad holder} to get H\"older estimates on the gradient of $u$ and then in Proposition \ref{prop:mu holder}, we get H\"older in time estimates on $\mu$. As a Corollary $H^\theta$ and $D_pH^\theta$ have H\"older regularity in time and space as well.
            \item Finally we use Lemma \ref{lem: full holder} to get full H\"older regularity of solutions $u$ to the fractional Hamilton Jacobi Equation.
        \end{enumerate}
        \item Section \ref{sub: sc} is devoted to showing the semi-concavity of solutions $u$ to the fractional Hamilton Jacobi Equation.
    \end{enumerate}

	\subsection{Estimates for the FP equation}\label{sub: FP est}
	
	Consider the fractional Fokker-Planck equation in the periodic setting
	
	\begin{equation}\label{ffpe}
		\begin{cases}
			\partial_t m+(-\Delta)^s m + \Div(bm)=0 &\mbox{in }Q_T\\
			m(x,0)=m_0(x)&\mbox{in }\mathbb{T}^d,
		\end{cases}
	\end{equation}
	
	where $ m_0\in L^\infty(\bb{T}^d)$ is a probability density. Given $b\in L^\infty(\bb{T}^d)$ such that $[\mbox{div} b]^-\in L^{\infty}(Q_T)$, a function $m\in L^2(0,T; H_2^s(\bb{T}^d))=\bb{H}_2^s(Q_T)$ with $\partial_t m\in L^2(0,T; H_2^{-1}(\bb{T}^d))=\bb{H}_2^{-1}(Q_T)$ is a weak solution to \eqref{ffpe} if for every $\varphi\in \s{C}^\infty(\bb{T}^d\times [0,T))$ one has
		
		$$\iint_{Q_T} \del{-m\partial_t \varphi -bm\cdot D\varphi +(-\Delta)^{\frac{s}{2}}m (-\Delta)^{\frac{s}{2}}\varphi} \dif x\dif t=\int_{\bb{T}^d} \varphi(x,0)m_0(x)\dif x.$$
	
	
	
	\begin{proposition}\label{prop:fp est}
		Let $m_0\in \s{C}^0(\bb{T}^d)$ and $b\in C_x^1(Q_T)$ such that $$\enVert{m_0}_\infty+\enVert{b}_\infty+\enVert{\sbr{\Div b}^-}_\infty\leq K $$
		
		Then there exists $C=C(K)$ such that for every weak solution $m$, it holds that:
		
		\begin{align}
			\label{eq: m bounds} \enVert{m}_{\infty;Q_T}&\leq C\\
			\label{eq: m s/2 bounds}\iint_{Q_T} [(-\Delta)^{s/2}m]^2 \dif x \dif t&\leq C\\
			 \enVert{\partial_t m}_{\bb{H}_2^{-1}(Q_T)}&\leq C
		\end{align}
	\end{proposition}
    \begin{proof}
        We follow the proof of Proposition 3.3 from \cite{cirant2019existence}. To show \eqref{eq: m bounds}, one can use standard comparison arguments with the function
        $$w(x,t):=m(x,t)e^{-(K+\varepsilon)t}-\enVert{m_0}_\infty$$
        with $\varepsilon\to0$ to yield 
        $$\enVert{m}_{\infty;Q_T}\leq \enVert{m_0}_\infty e^{KT}.$$
        Now by multiplying the equation in \eqref{ffpe} by $m$ and integrating over $Q_T$ we get
        $$\frac{1}{2}\int_0^T\od{}{t}\enVert{m}_{L^2(\bb{T}^d)}^2+\iint_{Q_T}m(-\Delta)^sm\dif x\dif t=-\iint_{Q_T}m\Div(bm)\dif x\dif t.$$
        Then by the properties of the fractional laplacian and integration by parts, 
        \begin{equation}
        \begin{split}
            \frac{1}{2}\int_0^T\od{}{t}\enVert{m}_{L^2(\bb{T}^d)}^2+\iint_{Q_T}\sbr{(-\Delta)^{\frac{s}{2}}m}^2\dif x\dif t&=\iint_{Q_T}mb\cdot Dm\dif x\dif t\\
            &=-\frac{1}{2}\iint_{Q_T}\del{\Div b}m^2\dif x\dif t.
        \end{split}
        \end{equation}
        Now using $[\Div(b)]^-\leq K$ and the above $L^\infty$ bound on $m$ yields
        $$\frac{1}{2}\enVert{m(T)}_{L^2(\bb{T}^d)}+\iint_{Q_T}\sbr{(-\Delta)^{\frac{s}{2}}m}^2\dif x\dif t\leq C(K)+\frac{1}{2}\enVert{m(0)}_{L^2(\bb{T}^d)},$$
        which gives us inequality \eqref{eq: m s/2 bounds}. Now from \eqref{ffpe} one can get
        \begin{equation*}
        \begin{split}
        \abs{\iint_{Q_T}\partial_tm\varphi\dif x\dif t}\leq \enVert{b}_{L^\infty(Q_T)}\enVert{m}_{L^2(Q_T)}\enVert{D\varphi}_{L^2(Q_T)}+\enVert{(-\Delta)^\frac{s}{2}m}_{L^2(Q_T)}\enVert{\varphi}_{\bb{H}^s_2(Q_T)}\\
        \leq C\enVert{\varphi}_{\bb{H}^1_2(Q_T)},
        \end{split}
        \end{equation*}
        and the last estimate follows.
    \end{proof}
    
\subsection{Estimates for the fixed point $\mu$}\label{sub: mu est}
The following Lemma from \cite{kobeissi2022mean} will be useful moving forward:

\begin{lemma}\label{lem: fixmu}
	Assume that $L$ satisfies \ref{hypo:Lstrconv}-\ref{hypo:Lbound}. For $t\in[0,T]$, $m\in \mathcal{P}(\mathbb{T}^d)$ and $p\in \s{C}^0(\mathbb{T}^d; \mathbb{R}^d )$, there exists a unique $\mu \in \mathcal{P}(\mathbb{T}^d \times \mathbb{R}^d)$ such that
	
	\begin{equation}
		\mu = \left(I_d, -D_pH(\cdot,p(\cdot),\mu)\right)\# m,
	\end{equation}
	
	Moreover, $\mu$ satisfies 
	
	\begin{align}
		\Lambda_{\tilde{q}}(\mu)^{\tilde{q}} &\leq 4C_0^2 +\dfrac{(\tilde{q})^{q-1} (2C_0)^q}{q}\|p\|_{L^q(m)}^q,\\
		\Lambda_\infty (\mu) &\leq C_0 \left( 1+\|p\|_\infty + \Lambda_{\tilde{q}}(\mu) \right). 
	\end{align}
\end{lemma}

We will see in the proof of Lemma \ref{lem: grad u bound}, that we also need a bound on the time integral of the moments of $\mu$.
\begin{lemma}\label{lem: moment int}
    Under assumptions \ref{hypo:Lstrconv}-\ref{hypo:Lbound} and \ref{hypo:Monogregx}, if $(u,m,\mu)$ is a solution to \eqref{eq:MFGCtheta} for $\theta\in(0,1]$, 
    $$\int_0^T\Lambda_{\tilde q}(\mu(t))^{\tilde q}\dif t\leq 2C_0^2\theta^{\tilde q}\del{1+T}.$$
\end{lemma}
\begin{proof}
    Our proof follows the strategy of the first step in \cite[Lemma 5.1]{kobeissi2022mean}.
    Define $(X(t),\alpha(t))$ as follows:
    \begin{equation}
        \begin{cases}
            \alpha(t)=-D_pH^\theta(X(t),Du(X(t)),\mu(t)),\\
            \dif X(t)=\alpha(t)\dif t+\dif J(t),\\
            X(0)\sim m_0
        \end{cases}
    \end{equation}
    where $J(t)$ is the L\'evy jump process generated by the fractional Laplacian.
    Note $\Lambda_{\tilde q}(\mu(t))^{\tilde q}=\int \abs{\alpha}^{\tilde q}\mu(t)(\dif x,\dif \alpha)=\bb{E}|\alpha(t)|^{\tilde q}$. Also, $(X(t),\alpha(t))$ minimizes
    \begin{equation}
        \bb{E}\sbr{\int_0^TL^\theta(X(t),\alpha(t),\mu(t))\dif t+\theta u_T(X(T))}.
    \end{equation}
    Now recall that for any $t\in[0,T]$, $m(t)$ is the law of $X(t)$ and $\mu(t)$ is the law of $(X(t),\alpha(t))$. We now introduce ${\tilde X}(t)$ as the stochastic process defined by
    \begin{equation}
        \begin{cases}
            \dif {\tilde X}(t)=\dif J(t)\\
            {\tilde X}(0)\sim m_0,
        \end{cases}
    \end{equation}
    and set ${\tilde m}(t)=\mathcal{L}({\tilde X}(t))$ and ${\tilde \mu}(t)=\mathcal{L}({\tilde X}(t))\times \delta_0,$ for $t\in[0,T].$ Now for $\mu^1,\mu^2\in\mathcal{P}(\bb{R}^d\times\bb{R}^d),$ define $I(\mu^1,\mu^2)$ as follows:
    $$I(\mu^1,\mu^2)=\int_0^T\int_{\bb{T}^d\times\bb{R}^d}L^\theta(x,\alpha,\mu^1(t))\dif \mu^2(t)(x,\alpha)\dif t $$
    Since 
    \begin{equation}
        \bb{E}\sbr{\int_0^TL^\theta(X(t),\alpha(t),\mu(t))\dif t+\theta u_T(X(T))}\leq\bb{E}\sbr{\int_0^TL^\theta({\tilde X}(t),0,\mu(t))\dif t+\theta u_T({\tilde X}(T))},
    \end{equation}
    \begin{equation}
        I(\mu,\mu)+\int \theta u_T(x)\dif m(T)(x)\leq I(\mu,{\tilde \mu})+\int\theta u_T(x)\dif {\tilde m}(T)(x).
    \end{equation}
    This yields 
    \begin{equation}\label{eq: I mu dif}
        I(\mu,\mu)-I(\mu,{\tilde \mu})\leq 2\theta\enVert{u_T}_\infty.
    \end{equation}
    Now integrating over $[0,T]$ the Lasry-Lions monotonicity condition on $L$, which is assumption \ref{hypo:LMono}, with $(\mu(t),{\tilde \mu}(t))$, gives
    \begin{equation}
        \begin{split}
            I({\tilde\mu},\mu)&\leq I(\mu,\mu)-I(\mu,{\tilde\mu})+I({\tilde\mu},{\tilde\mu})\\
            &\leq 2\theta\enVert{u_T}_\infty+I({\tilde \mu},{\tilde\mu}).
        \end{split}
    \end{equation}
    Then from \ref{hypo:Lbound}
    $$I({\tilde \mu},{\tilde\mu})=\int_0^T\int_{\bb{T}^d}\theta L(x,0,{\tilde\mu}){\tilde m}(t,\dif x)\dif t\leq C_0\theta T$$
    Also,
    \begin{equation}
        \begin{split}
            I({\tilde\mu},\mu)\geq \int_0^T\int_{\bb{T}^d\times\bb{R}^d}\del{C^{-1}_0\theta^{1-\tilde q }|\alpha|^{\tilde q}-C_0\theta}\dif \mu(t)(x,\alpha)\dif t\\
            =C_0^{-1}\theta^{1-\tilde q }\int_0^T\Lambda_{\tilde q}(\mu(t))^{\tilde q}\dif t-C_0\theta T.
        \end{split}
    \end{equation}
    Combining these results we see
    $$\int_0^T\Lambda_{\tilde q}(\mu(t))^{\tilde q}\dif t\leq 2C_0\theta^{\tilde q}\del{\enVert{u_T}_\infty+C_0T}.$$
    and using \ref{hypo:Monogregx} the desired result holds.
\end{proof}

    \subsection{Estimates for HJ equation}\label{sub: HJ est}
    We study the Fractional Hamilton Jacobi equation
    \begin{equation}\label{hjb}
		\begin{cases}
			-\partial_t u+(-\Delta)^s u +H^\theta(x,D u,\mu)=0 &\mbox{on }Q_T\\
			u(x,T)=\theta u_T(x)&\mbox{in }\mathbb{T}^d
		\end{cases}
	\end{equation}

\begin{lemma}\label{lem: u bdd}
    Suppose $L$ satisfies \ref{hypo:Lstrconv}-\ref{hypo:Lbound} and \ref{hypo:Monogregx} holds. If $(u,m,\mu)$ is a solution to \eqref{eq:MFGCtheta},
    \begin{equation}\label{est2}
	   \norm{u}_{\infty,Q_T} \leq C\theta,
    \end{equation}
    where $C$ is a finite constant that depends on $C_0$ and $T$.
\end{lemma}
\begin{proof}
    Since $u$ is a solution to \eqref{hjb}, a linearization argument and comparison principle yield
    \begin{equation}
	   \norm{u}_{\infty,Q_T} \leq 
        \theta\norm{u_T}_{\infty,\bb{T}^d}+C_0\theta\int_0^T\del{ 1+\theta^{-\tilde q}\Lambda_{\tilde q}(\mu(s))^{\tilde q}}\dif s
    \end{equation}
    Then using Lemma \ref{lem: moment int} and \ref{hypo:Monogregx},
    \begin{equation}
        \enVert{u}_{\infty,Q_T}\leq C_0\del{1+T+2C_0^2\del{1+T}}\theta
    \end{equation}
\end{proof}

\begin{lemma}\label{lem: grad u bound}
	Suppose $H$ satisfies assumption \ref{hypo: Hxbound} and \ref{hypo:Monogregx} holds. If $(u,m,\mu)$ is a solution of \eqref{eq:MFGCtheta}, 
    $$\norm{D u}_\infty\leq C\theta$$
    for some finite $C>0$ that depends on $\enVert{Du_T}_\infty$.
\end{lemma}	

\begin{proof}
    We have that $u$ satisfies
	\[-\partial_t u +(-\Delta)^s u +H^\theta(x,D u,\mu)=0.\]
	For $e\in\bb{R}^d$, $\vert e\vert=1$, taking the directional derivative with respect to $e$ yields
	\[-\partial_t \partial_e u +(-\Delta)^s \partial_e u +D_p H^\theta(x,D u,\mu)\cdot D\partial_e u + \partial_e H^\theta(x,D u, \mu)=0\]
    in the sense of distributions.
	Now set $v_e=\frac{1}{2}(\partial_e u)^2$. Then in the sense of distributions,
	$$-\partial_t v_e +(-\Delta)^s v_e +D_p H^\theta(x,D u,\mu)\cdot D v_e + \partial_e H^\theta(x,D u, \mu)\partial_e u=0$$
	
	Since $\abs{ D_xH^\theta\left(x,p,\mu\right)}
	\leq
	C\theta\del{1+\abs{p}}+C\theta^{1-\tilde q}\Lambda_{\tilde q}(\mu)^{\tilde q}$,
	\[-\partial_t v_e +(-\Delta)^s v_e +D_p H^\theta(x,D u,\mu)\cdot D v_e \leq C\theta\left( 1+\vert D u\vert^2 \right)+C\theta^{1-\tilde q}\Lambda_{\tilde q}(\mu)^{\tilde q}	
	\]
	
	Set $v=\frac{1}{2}\vert D u\vert^2=\frac{1}{2}\left(v_{e_1}+\cdots+v_{e_d}\right)$. Then,
	\[-\partial_t v +(-\Delta)^s v +D_p H^\theta(x,D u,\mu)\cdot D v \leq C\theta\left( 1+ v \right)+C\theta^{1-\tilde q}\Lambda_{\tilde q}(\mu)^{\tilde q}	
	\]	
	
	Let $w(x,t)=e^{-C\theta t}v(x,T-t)-C\theta t-C\theta^{1-\tilde q}\int_0^t\Lambda_{\tilde q}(\mu(s))^{\tilde q}\dif s$. Notice that $w(x,0)=\frac{1}{2}\theta\vert D u_T(x)\vert^2$ and
	\[\partial_t w +(-\Delta)^s w +D_p H^\theta(x,D u(x,T-t),\mu(T-t))\cdot D w \leq 0	
	\]	
	
	By the Maximum Principle for linear fractional equations \cite{garroni1992green}, $w\leq \max \frac{1}{2}\theta\vert D u_T\vert^2 $. Thus the result holds.
\end{proof}

Combining Lemmas \ref{lem: grad u bound} and \ref{lem: fixmu}, we get
\begin{corollary}\label{cor: moment est}
	Suppose $L$ satisfies \ref{hypo:Lstrconv}-\ref{hypo:Lbound}, $H$ satisfies \ref{hypo: Hxbound}, and \ref{hypo:Monogregx} holds. If $(u,m,\mu)$ is solution of \eqref{eq:MFGCtheta}, then
 $$\sup_{t\in[0,T]} \Lambda_{\tilde q} (\mu(t)) \leq C\theta$$
 and
 $$\sup_{t\in[0,T]} \Lambda_\infty (\mu(t)) \leq C\theta$$
 where $C$ is a constant.
\end{corollary}

\begin{lemma}\label{lem: H bdd}
        Suppose $L$ satisfies \ref{hypo:Lstrconv}-\ref{hypo:Lbound}, $H$ satisfies \ref{hypo: Hxbound}, and \ref{hypo:Monogregx} holds. If $(u,m,\mu)$ solves \eqref{eq:MFGCtheta},
        $$\enVert{H^\theta(x,Du,\mu)}_\infty\leq C\theta,$$
        and
        $$\enVert{D_pH^\theta(x,Du,\mu)}_\infty\leq C\theta,$$
        where $C$ is a finite constant that depends on $C_0, q, {\tilde q},$ and $\enVert{Du}_\infty$.
\end{lemma}
	\begin{proof}
	    By our assumptions on the Lagrangian $L$, we can apply Lemma \ref{lem:Hbound} to get
        $$\abs{H^\theta(x,Du,\mu)}\leq C_0\theta\del{1+\abs{Du}^q}+C_0\theta^{1-\tilde q}\Lambda_{\tilde q}(\mu)^{\tilde q}.$$
        Now by Corollary \ref{cor: moment est}, $\Lambda_{\tilde q}(\mu)^{\tilde q}\leq C\theta^{\tilde q}.$
        Thus
        $$\enVert{H^\theta(x,Du,\mu)}_\infty\leq C_0\del{1+\enVert{Du}_\infty^q+C}\theta.$$
        Similarly,
        \begin{equation*}
        \abs{D_pH^\theta(x,Du,\mu)}\leq C_0\theta\del{1+\abs{Du}^{q-1}}+C_0\Lambda_{\tilde q}(\mu).
        \end{equation*}
        Again using Corollary \ref{cor: moment est},
        $$\enVert{D_pH^\theta(x,Du,\mu)}_\infty \leq C_0\theta\del{1+\enVert{Du}_\infty^{q-1}+C}.$$
        Both results hold using that $\enVert{Du}_\infty\leq C$ from Lemma \ref{lem: grad u bound}.
	\end{proof}
    Combining Lemmas \ref{lem: grad u bound}, \ref{lem: H bdd}, and \ref{lem: grad holder}, we get
\begin{corollary}
    \label{cor: Du Cbeta}
    Suppose $L$ satisfies \ref{hypo:Lstrconv}-\ref{hypo:Lbound}, $H$ satisfies \ref{hypo: Hxbound}, and \ref{hypo:Monogregx} holds. If $(u,m,\mu)$ solves \eqref{eq:MFGCtheta}, then
    \begin{equation}
        \enVert{Du}_{\s{C}^{\beta,\frac{\beta}{2s}}} \leq C\theta,
    \end{equation}
    where $C > 0$ and $\beta \in (0,1)$ are constants.
\end{corollary}

We now prove the crucial time regularity of $\mu(t)$.    

\begin{proposition}\label{prop:mu holder}
    Suppose $L$ satisfies \ref{hypo:Lstrconv}-\ref{hypo:Lbound}, $s>\frac{1}{2}$, $H$ satisfies \ref{DpmuH}-\ref{DxxH}, and \ref{hypo:Monogregx} holds.
    If $(u,m,\mu)$ is a solution to \eqref{eq:MFGCtheta}, $\enVert{\mu}_{ \s{C}^\frac{\beta}{2}\del{[0,T],\mathcal{P}(\bb{T}^d\times\bb{R}^d)}}\leq C(1+\theta)$ for some finite constants $C>0$ and $\beta \in (0,1)$ that do not depend on $\theta$.
\end{proposition}

Proposition \ref{prop:mu holder} will follow from a more abstract lemma:
\begin{lemma}
    \label{lem:holder paths of measures}
    Let $b = b(x,t)$ be a bounded vector field, and let $m$ be the weak solution of
    \begin{equation}
        \partial_tm+(-\Delta)^sm+\Div(bm)=0.
    \end{equation}
    Then $m:[0,T] \to \s{P}_r(\bb{R}^d)$ is H\"older continuous with
    \begin{equation}\label{eq: WPm}
        W_r(m(t_1),m(t_0))\leq \enVert{b}_\infty|t_1-t_0|+C|t_1-t_0|^{1/2}.
    \end{equation}
    Now let $\mu(t)$ be the unique solution of
    \begin{equation}
        \mu(t) = (I,-D_p H^\theta(\cdot,Du(\cdot,t),\mu(t))_\sharp m(t)
    \end{equation}
    where $Du \in \s{C}^{\beta,\frac{\beta}{2s}}$.
    Then $\mu:[0,T] \to \s{P}_r(\bb{R}^d \times \bb{R}^d)$ is H\"older continuous with an estimate
    \begin{equation} \label{eq:mu holder}
        \enVert{\mu}_{\s{C}^\frac{\beta}{2}([0,T],\mathcal{P}(\bb{T}^d\times\bb{R}^d))}\leq C\del{1+\theta +\theta\enVert{Du}_{\s{C}^{\beta,\frac{\beta}{2s}}(Q_T)}}.
    \end{equation}
\end{lemma}
\begin{proof}
    Let $X(t)$ be the solution of the stochastic differential equation
    \begin{equation}
        X(t)-X(0)=\int_0^tb(X(\tau),\tau)\dif\tau+\int_0^t\int_{\bb{T}^d}j(X(\tau),\tau,z)\lambda(\dif z,\dif t),
    \end{equation}
    where $b(x,t) := -D_p H^\theta(x,Du(x,t),\mu(t))$, the jumps of the L\'evy process are given by $j(X(\tau),\tau,z)=z$, and $\lambda(t,\cdot)=\rho(t,\cdot)-t\nu(\cdot)$ where $\rho(t,\cdot)$ is a Poisson measure with L\'evy measure $\nu(\cdot)$ given by
    $$\nu(dz)=C(d,s)\frac{\dif z}{|z|^{d+2s}}$$
    Such an $X(t)$ is distributed according to $m(t)$.
    Then for some $q\geq 1$
    \begin{equation}
    \del{\bb{E}\abs{X(t_1)-X(t_0)}^r}^{1/r} 
        \leq \enVert{b}_{\infty}\abs{t_1-t_0}+\del{\bb{E}\abs{\int_{t_0}^{t_1}\int_{\bb{T}^d}j(X(t),t,z)\lambda(\dif z,\dif t)}^r}^{1/r}.
    \end{equation}
    Now from our above definition of the fractional laplacian along with estimates (4.49) of \cite{garroni1992green} 
    \begin{equation}
    \begin{split}
        \del{\bb{E}\abs{\int_{t_0}^{t_1}\int_{\bb{T}^d}j(X(t),t,z)\lambda(\dif z,\dif t)}^r}^{1/r}&\leq C_q\del{\bb{E}\abs{\int_{t_0}^{t_1}\dif t\int_{\bb{T}^d}|z|^{2-2s-d}\dif z}^{q/2}}^{1/q}\\
        &\leq C\abs{t_1-t_0}^{1/2}.
    \end{split}    
    \end{equation}
    Thus we have
    \begin{equation}
        W_r(m(t_1),m(t_0))\leq \enVert{b}_\infty|t_1-t_0|+C|t_1-t_0|^{1/2}.
    \end{equation}
    Now let $\bar{X}(t)=\del{X(t),-D_pH^\theta(X(t),Du(X(t),t),\mu(t))}$, so that $\bar{X}(t)\sim \mu(t)$. Then using Hypothesis \ref{DxxH} and Equation \eqref{eq:Hpbound} together with Lemma \ref{lem: grad u bound} and Corollary \ref{cor: moment est}, we get
    \begin{equation}\label{eq: Wqmu}
    \begin{split}
            \del{\bb{E}\abs{\bar{X}(t_1)-\bar{X}(t_0)}^r}^{1/r}\\
            &\hspace{-1in}\leq \del{\bb{E}\abs{{X}(t_1)-{X}(t_0)}^r}^{1/r}\\
            &\hspace{-.9in} +\del{\bb{E}\abs{D_pH^\theta(X(t_1),Du(X(t_1),t_1),\mu(t_1))-D_pH^\theta(X(t_0),Du(X(t_0),t_0),\mu(t_0))}^r}^{1/r}\\
            &\hspace{-1in}\leq \del{\bb{E}\abs{{X}(t_1)-{X}(t_0)}^r}^{1/r}+C\theta\del{\bb{E}\abs{{X}(t_1)-{X}(t_0)}^r}^{1/r}\\
            &\hspace{-.9in} +C\theta\del{\bb{E}\abs{Du(X(t_1),t_1)-Du(X(t_0),t_0)}^r}^{1/r}\\
            &\hspace{-.9in} +\del{\bb{E}\abs{D_pH^\theta(X(t_0),Du(X(t_0),t_0),\mu(t_1))-D_pH^\theta(X(t_0),Du(X(t_0),t_0),\mu(t_0))}^r}^{1/r}
    \end{split}        
    \end{equation}
    If $Du\in \s{C}^{\beta,\frac{\beta}{2s}}(Q_T)$, we may write 
    \begin{equation}\label{eq: holderDu}
        \del{\bb{E}\abs{Du(X(t_1),t_1)-Du(X(t_0),t_0)}^r}^{1/r}\leq \enVert{Du}_{\s{C}^{\beta,\frac{\beta}{2s}}(Q_T)}\del{|t_1-t_0|^\frac{\beta}{2s}+\del{\bb{E}\abs{X(t_1)-X(t_0)}^{\beta r}}^{1/r}}.
    \end{equation}
    Now by Assumption \ref{DpmuH} and using \eqref{eq: holderDu} and \eqref{eq: WPm} in \eqref{eq: Wqmu} yields
    \begin{multline}
        (1-C_R\theta)\del{\bb{E}\abs{\bar{X}(t_1)-\bar{X}(t_0)}^r}^{1/r}\leq C(1+C\theta)(\enVert{b}_\infty|t_1-t_0|+|t_1-t_0|^\frac{1}{2})\\
        +C\theta\enVert{Du}_{\s{C}^{\beta,\frac{\beta}{2s}}(Q_T)}\del{\abs{t_1-t_0}^\frac{\beta}{2s}+\del{\enVert{b}_\infty|t_1-t_0|+|t_1-t_0|^\frac{1}{2}}^\beta}.
    \end{multline}
    Since $W_r(\mu(t_1),\mu(t_0)) \leq \del{\bb{E}\abs{\bar{X}(t_1)-\bar{X}(t_0)}^r}^{1/r}$, we now have the estimate \eqref{eq:mu holder}.

\end{proof}

\begin{proof}[Proof of Proposition \ref{prop:mu holder}]
    Applying Corollary \ref{cor: Du Cbeta} in \eqref{eq:mu holder}, we get the estimate\\
    $\enVert{\mu}_{\s{C}^\frac{\beta}{2}([0,T],\mathcal{P}(\bb{T}^d\times\bb{R}^d))}\leq C(1+\theta)$. 
\end{proof}

\begin{corollary}\label{cor: H holder}
    Let $s>\frac{1}{2}$. Suppose $(u,m,\mu)$ is a solution to \eqref{eq:MFGCtheta} for $\theta\in[0,1]$, then if $L$ satisfies \ref{hypo:Lstrconv}-\ref{hypo:Lbound}, $H$ satisfies \ref{DpmuH}-\ref{DxxH}, and \ref{hypo:Monogregx} holds, then 
    $$\enVert{H^\theta(x,Du,\mu)}_{\s{C}^{\beta,\frac{\beta}{2}}(Q_T)}\leq C\theta,$$
    and
    $$\enVert{D_pH^\theta(x,Du,\mu)}_{\s{C}^{\beta,\frac{\beta}{2}}(Q_T)}\leq C\theta.$$
\end{corollary}
\begin{proof}
    Taking $x,y\in\bb{T}^d$ and applying Assumption \ref{hypo: Hxbound} as well as the results of Lemma \ref{lem:Hbound}
    \begin{equation}
        \begin{split}
            &\abs{H^\theta(y,Du(y,t),\mu(t))-H^\theta(x,Du(x,t),\mu(t))}\\
            &\leq \int_0^1 \abs{D_{x}H^\theta(x_\lambda,Du,\mu)|y-x|+D_{p}H^\theta(x,Du_\lambda,\mu)|Du(y,t)-Du(x,t)|}\dif\lambda\\
            &\leq \del{C\theta(1+\abs{Du})+C\theta^{1-\tilde q}\Lambda_{\tilde q}(\mu)^{\tilde q}}\abs{y-x}+\int_0^1\del{C_0\theta(1+|Du_\lambda|^{q-1})+C_0\Lambda_{\tilde q}(\mu)}|Du(y,t)-Du(x,t)|\dif \lambda\\
            &\leq C\theta|y-x|+C\theta[Du(\cdot,t)]_{\s{C}^\beta(\bb{T}^d)}|y-x|^\beta
        \end{split}
    \end{equation}
    Similarly, taking $\tau,t\in[0,T]$ we get
    \begin{equation}
    \begin{split}
         &\abs{H^\theta(x,Du(x,\tau),\mu(\tau))-H^\theta(x,Du(x,t),\mu(t))}       \\
         &\leq \int_0^1D_{p}H^\theta(x,\lambda Du(\tau)+(1-\lambda)Du(t),\mu(\tau))|Du(x,\tau)-Du(x,t)|\dif\lambda+C_R\theta W_q(\mu(\tau),\mu(t))\\
         &\leq C\theta\del{\sbr{Du(x,\cdot)}_{\s{C}^\frac{\beta}{2s}([0,T])}|\tau-t|^{\beta/2s}+\enVert{\mu}_{\s{C}^\frac{\beta}{2}([0,T],\mathcal{P}(\bb{T}^d\times\bb{R}^d)}|\tau-t|^{\beta/2}},
    \end{split}        
    \end{equation}
    So $$\enVert{H^\theta(x,Du,\mu)}_{\s{C}^{\beta,\frac{\beta}{2}}(Q_T)}\leq C\theta.$$

    Now taking $x,y\in\bb{T}^d$, we can see by Assumption \ref{DxxH}
    \begin{equation}
    \begin{split}
         &\abs{D_pH^\theta(y,Du(y,t),\mu(t))-D_pH^\theta(x,Du(x,t),\mu(t))}       \\
         &\leq \int_0^1 \abs{D^2_{px}H^\theta(x_\lambda,Du,\mu)|y-x|+D^2_{pp}H^\theta(x,Du_\lambda,\mu)|Du(y,t)-Du(x,t)|}\dif\lambda\\
         &\leq C\theta|y-x|+C\theta\enVert{Du(\cdot,t)}_{\s{C}^\beta(\bb{T}^d)}|y-x|^\beta\\
         &\leq C\theta\del{1+\enVert{Du(\cdot,t)}_{\s{C}^\beta(\bb{T}^d)}}|y-x|^\beta.
    \end{split}
    \end{equation}
    So $\enVert{D_pH^\theta(\cdot,Du(\cdot,t),\mu(t))}_{\s{C}^\beta(\bb{T}^d)}\leq C\theta\del{1+\enVert{Du(\cdot,t)}_{\s{C}^\beta(\bb{T}^d)}}.$ Similarly, taking $\tau,t\in[0,T]$ we get
    \begin{equation}
    \begin{split}
         &\abs{D_pH^\theta(x,Du(x,\tau),\mu(\tau))-D_pH^\theta(x,Du(x,t),\mu(t))}       \\
         &\leq \int_0^1D^2_{pp}H^\theta(x,\lambda Du(\tau)+(1-\lambda)Du(t),\mu(\tau))|Du(x,\tau)-Du(x,t)|\dif\lambda+\theta W_q(\mu(\tau),\mu(t))\\
         &\leq C\theta\enVert{Du(x,\cdot)}_{\s{C}^\frac{\beta}{2s}([0,T])}|\tau-t|^{\beta/2s}+\theta\enVert{\mu}_{\s{C}^\frac{\beta}{2}([0,T],\mathcal{P}(\bb{T}^d\times\bb{R}^d)}|\tau-t|^{\beta/2},
    \end{split}        
    \end{equation}
    so that $\enVert{D_pH^\theta(x,Du(x,\cdot),\mu(\cdot))}_{\s{C}^\frac{\beta}{2}([0,T])}\leq C\theta\del{\enVert{Du(x,\cdot)}_{\s{C}^\frac{\beta}{2s}([0,T])}+\enVert{\mu}_{\s{C}^\frac{\beta}{2}([0,T],\mathcal{P}(\bb{T}^d\times\bb{R}^d)}}.$
    Now
    $$\enVert{D_pH^\theta(x,Du,\mu)}_{\s{C}^{\beta,\frac{\beta}{2}}(Q_T)}\leq C\theta\del{1+\enVert{Du}_{\s{C}^{\beta,\frac{\beta}{2}}(Q_T)}+\enVert{\mu}_{\s{C}^\frac{\beta}{2}([0,T],\mathcal{P}(\bb{T}^d\times\bb{R}^d)}}.$$
    Using Proposition \ref{prop:mu holder} and Lemmas \ref{lem: grad holder} and \ref{lem: H bdd} the result holds.
\end{proof}

\begin{corollary}\label{cor: u holder}
    Let $\beta\in(0,1)$ and $s>\frac{1}{2}$ so that $2s+\beta$ is not an integer. Suppose $(u,m,\mu)$ is a solution to \eqref{eq:MFGCtheta}. If $L$ satisfies \ref{hypo:Lstrconv}-\ref{hypo:Lbound}, $H$ satisfies \ref{DpmuH}-\ref{DxxH}, and \ref{hypo:Monogregx} holds, then
    $$\enVert{(-\Delta)^su}_{\s{C}^{\beta,\frac{\beta}{2}}(Q_T)}+\enVert{\partial_tu}_{\s{C}^{\beta,\frac{\beta}{2}}(Q_T)}\leq C\del{\theta\enVert{u_T}_{\s{C}^{2s+\beta}(\bb{T}^d)}+\enVert{H^\theta(x,Du,\mu)}_{\s{C}^{\beta,\frac{\beta}{2}}(Q_T)}}\leq C\theta.$$
\end{corollary}
\begin{proof}
    This follows from Lemma \ref{lem: full holder} and Corollary \ref{cor: H holder}.
\end{proof}

	\subsection{Semi-concavity for the HJ equation}\label{sub: sc}

	We analyze semi-concavity properties of solutions of
	
	\begin{equation}\label{hjb1}
		\begin{cases}
			-\partial_t u+(-\Delta)^s u +H^\theta(x,D u,\mu)=0 &\mbox{on }Q_T\\
			u(x,T)=\theta u_T(x)&\mbox{in }\mathbb{T}^d
		\end{cases}
	\end{equation}
	
	We prove that $u$ is semi-concave with constant that depends on the data.

		

    \begin{lemma}\label{lem: semi-con}
        For any $C>0$ there is a constant $C_1=C_1(C)$ such that if $L:[0,T]\times\bb{R}^d\times\bb{R}^d\times\mathcal{P}\del{\bb{R}^d\times\bb{R}^d}\to\bb{R}$ and $u_T:\bb{R}^d\to\bb{R}$ are continuous and such that
        \begin{equation}
            \enVert{L^\theta(\cdot,\alpha,\mu)}_{\s{C}^2}\leq C\theta \ \forall \alpha,\mu \ \text{s.t.}~\abs{\alpha} \leq C, \Lambda_\infty(\mu) \leq C, \qquad \text{and} \qquad\enVert{u_T}_{\s{C}^2}\leq C,
        \end{equation}
        then for each $\theta\in(0,1]$, \eqref{hjb1} has a unique bounded continuous viscosity solution with bounded gradient $D_x u$ which is given by the representation formula:
        \begin{equation}
            u(x,t)=\inf_{\alpha\in L^{\tilde{q}}([t,T],\bb{R}^d)}\bb{E}\sbr{\int_t^T L^\theta(x(\tau),\alpha(\tau),\mu(\tau))\dif \tau+\theta u_T(x(T))}
        \end{equation}
        where trajectories $x(\tau)$ are given by
        \begin{equation}
        \begin{cases}
           \dif x(\tau)=\alpha(\tau)\dif \tau+\dif J(\tau)\\
           x(t)=x.
        \end{cases}
        \end{equation}
        Here $J(\tau)$ is a $d$-dimensional, $2s$-stable pure jumps L\'evy process with associated L\'evy measure given by $\nu(\dif z)=C(d,s)\frac{\dif z}{|z|^{d+2s}}.$
        
        Moreover, $u$ satisfies the concavity estimate
        \begin{equation}\label{eq: sc}
             D^2_{xx}u\leq C_1\theta I_d    
        \end{equation}
        in the sense of distributions.
     \end{lemma}
     \begin{proof}
        First note that the solution $u$ to \eqref{hjb1} has a unique bounded continuous viscosity solution with bounded graident by \cite[Theorm 3]{droniou2006fractal}; see also \cite[Corollary 3.1]{barles2015regularity}. (For an overview of viscosity solutions, see e.g.~\cite{crandall1992user}).
        We turn to the semiconcavity estimate.
        Our proof uses a similar strategy as in \cite[Section 4.1]{cardaliaguet2010notes}.
          Now by the Markov property, the value function
         \begin{equation}
            v(x,t)=\inf_{\alpha\in L^{\tilde{q}}([t,T],\bb{R}^d)}\bb{E}\sbr{\int_t^T L^\theta(x(\tau),\alpha(\tau),\mu(\tau))\dif \tau+\theta u_T(x(T))}
        \end{equation} 
        satisfies the dynamic programming principle, and thus $v$ is a bounded continuous viscosity solution to \eqref{hjb}. Therefore, $u=v.$
        
        Now we show the semi-concavity of $u$ with respect to the $x$ variable. Let $x,y,t\in\bb{R}^d\times\bb{R}^d\times[0,T],$ $\lambda\in(0,1),$ $x_\lambda=\lambda x+(1-\lambda)y.$ Also let $\alpha\in L^{\tilde{q}}([t,T],\bb{R}^d)$ be $\varepsilon$-optimal for $u(x_\lambda,t)$ and set $x_\lambda(s)=x_\lambda+\int_s^t\alpha(\tau)\dif \tau+\int_s^t\dif J(s).$ Note that for any function $f$ such that $\enVert{f}_{\s{C}^2(\bb{R}^d)}\leq C$, the following holds
         \begin{equation}\label{eq: lambdaf}
         \begin{split}
             &\lambda f(x_\lambda(s)+x-x_\lambda)+(1-\lambda)f(x_\lambda(s)+y-x_\lambda) \\
             &\leq\lambda f(x_\lambda(s))+\lambda(x-x_\lambda)\int_0^1D_xf(\theta x_\lambda(s)+\theta(x-x_\lambda))\dif \theta\\
             &+(1-\lambda) f(x_\lambda(s))+(1-\lambda)(y-x_\lambda)\int_0^1D_xf(\theta x_\lambda(s)+\theta(y-x_\lambda))\dif \theta\\
             &= f(x_\lambda(s))+\lambda(1-\lambda)(x-y)\int_0^1\sbr{D_xf(\theta x_\lambda(s)+\theta(x-x_\lambda))-D_xf(\theta x_\lambda(s)+\theta(y-x_\lambda))}\dif \theta\\
             &\leq f(x_\lambda(s))+C\lambda(1-\lambda)|x-y|^2
         \end{split}
         \end{equation}        
         Then applying the inequality from \eqref{eq: lambdaf} to $L$ and $u_T$ 
         \begin{equation}
         \begin{split}
             \lambda u(x,t)&+(1-\lambda)u(y,t)\\
             \leq& \lambda\bb{E}\sbr{\int_t^TL^\theta(x_\lambda(s)+x-x_\lambda,\alpha(s),\mu(s))\dif s+\theta u_T(x_\lambda(T)+x-x_\lambda)}\\
             &+(1-\lambda)\bb{E}\sbr{\int_t^TL^\theta(x_\lambda(s)+y-x_\lambda,\alpha(s),\mu(s))\dif s+\theta u_T(x_\lambda(T)+y-x_\lambda)}\\
             \leq&\bb{E}\sbr{\int_t^TL^\theta(x_\lambda(s),\alpha(s),\mu(s))\dif s+\theta u_T(x_\lambda(T))}+C\theta\lambda(1-\lambda)|x-y|^2\\
             \leq &u(x_\lambda,t)+\varepsilon+C\theta\lambda(1-\lambda)|x-y|^2.
         \end{split}
         \end{equation}
         Thus $u$ is semi-concave with respect to $x$ and therefore satisfies \eqref{eq: sc}.
     \end{proof}

\section{Existence and uniqueness} \label{sec:existence}

For the reader's convenience, we recall the Leray-Schauder fixed point theorem:
\begin{theorem}[Leray-Schauder]
    Let $X$ be a Banach space and let $T: X \times [0,1] \rightarrow X$ be a continuous and compact mapping. Assume there exist $x_0 \in X$ and $C > 0$ so that $T(x,0) = x_0$ for all $x \in X$ and $\|x\|_X < C$ for all $(x,\tau) \in X \times [0,1]$ such that $T(x,\tau) = x$. Then there exists $x \in X$ such that $T(x,1) = x$.
\end{theorem}

We are now ready to state our result on existence of solutions.
\begin{theorem}[Main result]
	Let $s\in(\frac{1}{2},1)$. Suppose Assumptions \ref{hypo:Lstrconv}-\ref{hypo:Lbound}, \ref{DpmuH}-\ref{DxxH}, and \ref{hypo:Monogregx} hold. Then there exists a solution $(u,m,\mu)$ to \eqref{eq:MFG}.
\end{theorem}

\begin{proof} We plan to apply the Leray-Schauder theorem to prove existence of solutions. However, $\mathcal{P}(\bb{T}^d)$ is not a Banach space, so we follow the same strategy as in Section 6.1 of \cite{kobeissi2022mean}.
	Consider the map $\rho: \s{C}^0\left([0,T]; H_2^\varepsilon(\bb{T}^d)  \right) \rightarrow \s{C}^0\left([0,T]; H_2^\varepsilon(\bb{T}^d)  \right)$ for some $\varepsilon>0$ defined by
	\[ \rho(\tilde{m})(x,t) =\dfrac{\tilde{m}_{+}(x,t) -\int \tilde{m}_+(y,t)dy  }{ \max \left(1, \int \tilde{m}_{+}(y,t)dy \right) } + 1, \]
	where $\tilde{m}_{+}(x,t)=\max \left(0, \tilde{m}(x,t)\right)$. Also, $\tilde{m}^0$ is defined as the the unique weak solution of
	\[ \partial_t \tilde{m}^0  + (-\Delta)^s\tilde{m}^0 =0, \mbox{ on }(0,T)\times \bb{T}^d, \qquad \mbox{and }\tilde{m}^0(0, \cdot)=m_0.\]
	
	Now we construct our map $\Psi$ as follows. Let $\theta\in[0,1]$, $u\in \s{C}^{1,0}\left( \bb{T}^d\times [0,T]; \bb{R}  \right)$ and $\tilde{m}\in \s{C}^0\left([0,T]; H_2^\varepsilon(\bb{T}^d)  \right) $ for some $\varepsilon>0$ to be specified later.
	
	Define a probability measure $m= \rho(\tilde{m}+ \tilde{m}^0)$ and $(\mu, \alpha)\in \s{C}^0\left([0,T]; \mathcal{P}\left(\bb{T}^d\times \bb{R}^d\right)  \right)\times \s{C}^0\left( \bb{T}^d\times[0,T]; \bb{R}  \right)$ by
	\begin{align*}
		\alpha(x,t)&= -D_pH^\theta \left( x, D u(x,t), \mu(t)  \right), \\
		\mu(t) &= \left(I_d, \alpha(\cdot,t)\right)\# m(t)
	\end{align*}
	when $\theta>0$ and for $\theta=0$ we simply take $\alpha=0$ and $\mu(t)=m(t)\times\delta_0$.
	Recall from Corollary \ref{cor: moment est}, 
	\[\sup_{t\in[0,T]}\Lambda_{\tilde{q}}(\mu(t)) \leq C\theta.\]

    Define $\bar{u}\in \s{C}^{1,0}(\bb{T}^d\times[0,T];\bb{R})$ to be the unique viscosity solution (by \cite[Theorem 3]{droniou2006fractal}) of
    \begin{equation}
    -\partial_t\bar{u}+(-\Delta)^s\bar{u} + H^\theta(x,D \bar u,\mu(t)) = 0,
    \end{equation}
    with terminal condition $\bar u(x,T)=\theta u_T(x).$ We see that $\bar{u}$ is also the solution of a fractional parabolic equation with bounded right-hand side $-H^\theta(x,D \bar u,\mu(t))$, since by Assumption \eqref{eq:Hbound} we have
    \[\left| H^\theta(x,D\bar u,\mu)  \right|\leq C_0 \theta \left(1+|D\bar u|^{q}\right) + C_0 \Lambda_{\tilde{q}}(\mu)^{\tilde q} \leq C\theta,\]
    so by Lemma \ref{lem: grad holder},
    \begin{equation}\label{eq: Du holder}
        \enVert{D\bar{u}}_{\s{C}^{\beta,\frac{\beta}{2s}}(Q_T)}\leq C\theta
    \end{equation}
    for some $\beta\in(0,\beta_0)$.	
	
	Set $\bar\alpha = -D_pH^\theta(x,D \bar u,\mu(t))$. Note that by Assumption \eqref{eq:Hpbound},
    $$\abs{\bar \alpha}=\abs{D_pH^\theta(x,D\bar u,\mu)}\leq C_0 \theta \left(1+|D\bar u|^{q-1}\right) + C_0 \Lambda_{\tilde{q}}(\mu) \leq C\theta,$$
    Define $\bar{m}$ as the solution in sense of distributions of
    \begin{equation}
	\partial_t \bar{m} + (-\Delta)^s\bar{m} + \Div(\bar{\alpha} \bar{m}) =0,         
    \end{equation}
	with initial condition $\bar{m}(0, \cdot)= m_0$, where $\enVert{m_0}_{s-1+\varepsilon,2}\leq C$. By Proposition \ref{prop:fp est} if $\bar\alpha\in C_x^1(Q_T)$ such that 
 \begin{equation}
     \enVert{m_0}_\infty+\enVert{\bar\alpha}_\infty+\enVert{[\Div \bar\alpha]^-}_\infty\leq K,
 \end{equation}
then  $\bar{m}\in\bb{H}^s_2(Q_T)$ and $\partial_t\bar{m}\in \bb{H}_2^{-1}(Q_T)$. 
 Notice from the bounds on $D^2_{pp}H^\theta$ and $D^2_{px}H^\theta$, the convexity estimate $D^2_{pp} H^\theta \geq 0$, the semi-concavity estimate $D_{xx}^2 \bar u \leq C$ from Lemma \ref{lem: semi-con}, we have
 \begin{equation}\label{eq: divDpH}
          \Div\del{D_pH^\theta(x,D \bar u,\mu)} = \operatorname{tr}\del{D_{px}^2 H^\theta(x,D \bar u,\mu) + D_{pp}^2 H^\theta(x,D \bar u,\mu)D_{xx}^2 \bar u}\leq C\theta,
 \end{equation}
 for a finite constant $C>0$, so Proposition \ref{prop:fp est} holds, yielding $\enVert{\bar{m}}_{\mathcal{H}^{2s-1}_2(Q_T)}\leq\enVert{\bar{m}}_{\bb{H}_2^{s}(Q_T)}+\enVert{\partial_t\bar{m}}_{\bb{H}^{-1}_2(Q_T)}\leq C_K$. 
 Now $\enVert{m_0}_{s-1+\varepsilon_0,2}\leq C$ from \ref{hypo:Monogregx}. By Theorem \ref{thm: embedding} 
 \begin{equation}\label{eq: m holder}
 \enVert{\bar{m}}_{C^{\frac{\beta'}{s}-\frac{1}{2}}\del{[0,T];H^{2s-1-2\beta'}_2(\bb{T}^d)}}\leq C,
 \end{equation}
 for $\frac{s}{2}<\beta<s$.
 \begin{comment}
 They cite two theorems (thm 2.1 in 33, thm 6.29 in 37). to get that 
	$m\in \s{C}^{\frac{\beta}2,\beta}\left([0,T]\times\bb{T}^d\right)$
	for $\beta\in(0,\beta_0)$,
	and that its associated norm can be estimated from
	above by a constant which depends on
	$\lVert{\nabla_xu}\rVert_\infty$, $\beta$, $a$
	and the constants in the assumptions. Similar argument works with $\tilde{m}^0$ too, this is, $\tilde{m}^0\in \s{C}^{\frac{\beta}2,\beta}\left([0,T]\times\bb{T}^d\right)$. 
\end{comment}

Finally, take $\bar{\mu}(t)=\del{I_d,\bar\alpha(\cdot,t)}\#\bar{m}(t)$ for every $t\in [0,T]$.

 We can now define the map $\Psi: (\theta,u,\tilde{m})\mapsto(\bar{u},\bar{m}-\tilde{m}^0)$, from $\s{C}^{1,0}(\bb{T}^d\times[0,T];\bb{R})\times \s{C}^0\left([0,T]; H_2^\varepsilon(\bb{T}^d)  \right)$ onto itself for each $\theta\in\sbr{0,1}$.
 Suppose $(\theta_n,u_n,\tilde m_n)$ is a bounded sequence in $[0,1]\times\s{C}^{1,0}(\bb{T}^d\times[0,T];\bb{R})\times\s{C}^0\left([0,T]; H_2^\varepsilon(\bb{T}^d)  \right)$ and let $\Psi(\theta_n,u_n,\tilde m_n)=(\bar u_n,\bar m_n-\tilde m^0)$. For compactness, we must show $(\bar u_n,\bar m_n-\tilde m^0)$ converges to some $(\bar u,\bar m-\tilde m^0)\in \s{C}^{1,0}(\bb{T}^d\times[0,T];\bb{R})\times \s{C}^0\left([0,T]; H_2^\varepsilon(\bb{T}^d)  \right)$ up to a subsequence. Now from \eqref{eq: Du holder} and since $\theta_n$ is uniformly bounded,
 $$\enVert{D\bar u_n}_{\s C^{\beta,\frac{\beta}{2}}(Q_T)}\leq C\theta_n\leq C.$$
 Then by the Arzela-Ascoli theorem and uniform convergence of the gradient, there exists some $\bar u$ such that by passing to a subsequence $\bar u_n\to \bar u$ and $D (\bar u_n)\to D\bar u$ uniformly. Therefore $\bar u_n\to \bar u$ in $\s C^{1,0}(Q_T,\bb{R})$.

Now by Equation \eqref{eq: m holder}, we get uniform equicontinuity with respect to time of $\bar m_n$. Also, for $0<\varepsilon<2s-1-2\beta'<1$, $H^{2s-1-2\beta'}_2(\bb{T}^d)$ is relatively compact in $H^\varepsilon_2(\bb{T}^d)$. This can be seen by applying Corollary 2.4 of \cite{bellido25} with the Rellich-Kondrachov Theorem for Sobolev spaces. $W^{1,2}(\bb{T}^d)$ is compactly embedded in $L^2(\bb{T}^d)$, so the embedding 
$$H^{2s-1-2\beta'}_2(\bb{T}^d)=[W^{1,2}(\bb{T}^d),L^2(\bb{T}^d)]_{2s-1-2\beta'}\hookrightarrow[W^{1,2}(\bb{T}^d),L^2(\bb{T}^d)]_\varepsilon=H^\varepsilon_2(\bb{T}^d)$$
is compact. Thus by Lemma 1 of \cite{simon86}, there exists some $\bar m$ such that, up to passing to a subsequence, $\bar m_n\to \bar m$ and $\bar m-\tilde m^0\in \s C^0\del{[0,T];H^\varepsilon_2(\bb{T}^d)}$. Therefore the map $\Psi$ is compact.

For continuity of the map, assume $(\theta_n,u_n,\tilde m_n)\to(\theta,u,\tilde m)$ in $[0,1]\times\s{C}^{1,0}(\bb{T}^d\times[0,T];\bb{R})\times\s{C}^0\left([0,T]; H_2^\varepsilon(\bb{T}^d)  \right)$. It suffices to show that, up to passing to a subsequence, $\Psi(\theta_n,u_n,\tilde m_n)\to\Psi(\theta,u,\tilde m)$. From the previous argument, we may assume $\Psi(\theta_n,u_n,\tilde m_n)\to(\bar u,\bar m-\tilde m^0)$. Now set $\mu_n=\mu_{\theta_n,u_n,\tilde m_n}=(I_d,-D_pH^{\theta_n}(x,Du_n(x),\mu_n)\# m_n$ and $\mu=\mu_{\theta,u,\tilde m}=(I_d,-D_pH^{\theta}(x,Du(x),\mu)\# m$, where $m_n=\rho(\tilde m_n+\tilde m^0)$ and $m=\rho(\tilde m+\tilde m^0)$. By the continuity of the map $\rho$, we may say $m_n\to m$ in $\s P(\bb{T}^d)$. We aim to show $\mu_n\to\mu$ in $\s P(\bb{T}^d\times \bb{R}^d)$ endowed with the $r$-Wasserstein distance. 
 Now let $\bar{X}_n=\del{X_n,-D_pH^{\theta_n}(X_n,Du_n(X_n),\mu_n)}$ and $\bar{X}=\del{X,-D_pH^{\theta}(X,Du(X),\mu)}$, so that $\bar{X}_n\sim \mu_n$ and $\bar{X}\sim \mu$, where $X_n\sim m_n$ and $X\sim m$. Then using  \ref{DpmuH}, \ref{DxxH}, and \eqref{eq:Hpbound}, we get
    \begin{equation}
    \begin{split}
            \del{\bb{E}\abs{\bar{X}_n-\bar{X}}^r}^{1/r}\\
            &\hspace{-1in}\leq \del{\bb{E}\abs{{X}_n-{X}}^r}^{1/r} +\del{\bb{E}\abs{D_pH^{\theta_n}(X_n,Du_n(X_n),\mu_n)-D_pH^\theta(X,Du(X),\mu)}^r}^{1/r}\\
            &\hspace{-1in}\leq \del{\bb{E}\abs{{X}_n-{X}}^r}^{1/r}+C|\theta_n-\theta|+C\theta\del{\bb{E}\abs{{X}_n-{X}}^r}^{1/r}\\
            &\hspace{-.9in} +C\theta\del{\bb{E}\abs{Du_n(X_n)-Du(X_n)}^r}^{1/r}+C\theta\del{\bb{E}\abs{Du(X_n)-Du(X)}^r}^{1/r}+C_R'\theta W_r(\mu_n,\mu).
    \end{split}        
    \end{equation}
    So
    \begin{equation*}
        W_r(\mu_n,\mu)\leq C\del{\del{\bb{E}\abs{{X}_n-{X}}^r}^{1/r}+C|\theta_n-\theta|+\del{\bb{E}\abs{Du_n(X_n)-Du(X_n)}^r}^{1/r}+\del{\bb{E}\abs{Du(X_n)-Du(X)}^r}^{1/r}}.
    \end{equation*}
    Therefore, up to a subsequence, $W_r(\mu_n,\mu)\to 0$. Then,
    $$|H^{\theta_n}(x,D\bar{u}_n,\mu_n)-H^{\theta}(x,D\bar{u},\mu)|\to 0,$$
     and 
    $$|D_pH^{\theta_n}(x,D\bar{u}_n,\mu_n)-D_pH^{\theta}(x,D\bar{u},\mu)|\to 0.$$     
    Therefore $(\bar u,\bar m-\tilde m^0)$ is a solution to the given $(\theta,u,\tilde m)$ as desired, proving continuity.
    
 This map 
 satisfies $\Psi(0,u,\tilde{m})=(0,\tilde{m}^0)$ for any $(u,\tilde{m})$. Indeed, the fact that $\enVert{\alpha}_\infty\leq C\theta$, implies that as $\theta$ tends to $0$, $\bar{m}$ tends to $\tilde{m}^0$ and $\bar{u}$ tends to $0$. Therefore $\Psi$ is continuous at $\theta=0$.\\
 Moreover, if $\Psi(\theta,u,\tilde{m})=(u,\tilde{m})$, then $(u,\tilde{m})$ is uniformly bounded in $\s{C}^{1,0}(\bb{T}^d\times[0,T];\bb{R})\times \s{C}^0\left([0,T]; H_2^\varepsilon(\bb{T}^d)  \right)$. To see this note that from Lemma \ref{lem: u bdd} we get that if $u$ is a fixed point then
 \begin{equation}
	\norm{u}_{\infty,Q_T} \leq C\theta.
 \end{equation}
 Also, from Lemma \ref{lem: grad u bound} we get 
 \begin{equation}
     \enVert{D u}_\infty\leq C\theta.
 \end{equation}
 Therefore we get uniform bounds on $u\in \s{C}^{1,0}(\bb{T}^d\times[0,T];\bb{R})$. Now by Proposition \ref{prop:fp est}, $\enVert{\tilde m}_{\infty,Q_T}\leq C$ where $C$ depends on bounds on $m_0$, $D_pH^\theta(x,D u,\mu)$, and $\sbr{ \Div D_pH^\theta(x,D u,\mu)}^-$. By Lemma \ref{lem: H bdd},
 \begin{equation}
         \enVert{D_pH^\theta(x,D u,\mu)}_\infty\leq C\theta.
 \end{equation}

 That with inequality \eqref{eq: divDpH} yields
 \begin{equation}
     \enVert{m_0}_\infty +\enVert{D_pH^\theta(x,D u,\mu)}_\infty+\enVert{\sbr{ \Div D_pH^\theta(x,D u,\mu)}^-}_\infty\leq K,
 \end{equation}
 so the results of \ref{prop:fp est} hold. Thus if $\Psi(\theta,u,\tilde{m})=(u,\tilde{m})$, then $(u,\tilde{m})$ is uniformly bounded in $\s{C}^{1,0}( \bb{T}^d\times[0,T];\bb{R})\times \s{C}^0\left([0,T]; H_2^\varepsilon(\bb{T}^d)  \right)$.

 Thus, the fixed points of $\Psi(\theta)$ give the solutions to \eqref{eq:MFGCtheta}, $(u,\bar{m})$, where $\bar{m}=\rho(\tilde{m}+\tilde{m}^0)$, which are uniformly bounded. So by the Leray-Schauder fixed point theorem, there exists a solution to \eqref{eq:MFG}.
\end{proof}

\begin{theorem}[Uniqueness]
    There exists at most one solution of System \eqref{eq:MFG}.
\end{theorem}

\begin{proof}
    The proof is essentially the same as the one found in \cite{kobeissi2022mean}.
    Let $(u_i,m_i,\mu_i)$, $i = 1,2$ be two solutions.
    The key is that $u_1 - u_2$ can be validly used as a test function in the Fokker-Planck equation satisfied by both $m_1$ and $m_2$.
    Subtracting the results, we get
    \begin{equation*}
        \begin{aligned}
        0 =& \int_0^T\int_{\bb{T}^d} m_1(H(t,x,D_xu_1,\mu_1) - H(t,x,D_xu_2,\mu_2) + D_x(u_2 - u_1) \cdot D_pH(t,x,D_xu_1,\mu_1)) \dif x \dif t \\
        &+ \int_0^T\int_{\bb{T}^d} m_2(H(t,x,D_xu_2,\mu_2) - H(t,x,D_xu_1,\mu_1) + D_x(u_1 - u_2) \cdot D_pH(t,x,D_xu_2,\mu_2)) \dif x \dif t.
        \end{aligned}
    \end{equation*}
    Note that for $i = 1,2$, there exists a unique vector field $\alpha^{\mu_i}$ such that
    $$L(x,\alpha^{\mu_i},\mu_i) = D_xu_i \cdot D_pH(x,D_xu_i,\mu_i) - H(x,D_xu_i,\mu_i)$$
    and
    $$D_xu_i = -D_\alpha L(x,\alpha^{\mu_i},\mu_i).$$
    Thus,
    $$
    \begin{aligned}
        0 \leq& \int_0^T\int_{\bb{T}^d} m_1(L(x,\alpha^{\mu_2},\mu_2) - L(x,\alpha^{\mu_1},\mu_1) + D_\alpha L(x,\alpha^{\mu_2},\mu_2) \cdot (\alpha^{\mu_1}-\alpha^{\mu_2})) \dif x \dif t \\
        &+ \int_0^T\int_{\bb{T}^d} m_2(L(x,\alpha^{\mu_1},\mu_1) - L(x,\alpha^{\mu_2},\mu_2) + D_\alpha L(x,\alpha^{\mu_1},\mu_1) \cdot (\alpha^{\mu_2}-\alpha^{\mu_1})) \dif x \dif t.
    \end{aligned}
    $$
    Since $L$ is strictly convex,
    \begin{equation}
        L(x,\alpha_1,\mu) - L(x,\alpha_2,\mu) + D_\alpha L(x,\alpha_1,\mu) \cdot (\alpha_2 - \alpha_1) \leq 0
        \label{Eq: Convexity of L}
    \end{equation}
    with equality holding if and only if $\alpha_1 = \alpha_2$. Hence,
    $$
    \begin{aligned}
        0 \leq& \int_0^T\int_{\bb{T}^d} \bigg(m_1(L(x,\alpha^{\mu_1},\mu_2) - L(x,\alpha^{\mu_1},\mu_1)) + m_2(L(x,\alpha^{\mu_2},\mu_1) - L(x,\alpha^{\mu_2},\mu_2))\bigg)\dif x \dif t \\
        =& -\int_0^T\int_{\bb{T}^d \times \bb{R}^d} (L(x,\alpha,\mu_1) - L(x,\alpha,\mu_2)) \dif (\mu_1 - \mu_2)(x,\alpha)\dif t
    \end{aligned}
    $$
    By Assumption \ref{hypo:LMono}, this gives
    $$\int_0^T\int_{\bb{T}^d} \bigg(m_1(L(x,\alpha^{\mu_2},\mu_1) - L(x,\alpha^{\mu_1},\mu_1)) + m_2(L(x,\alpha^{\mu_1},\mu_2) - L(x,\alpha^{\mu_2},\mu_2))\bigg)\dif x \dif t = 0.$$
    By the condition for equality for \eqref{Eq: Convexity of L}, we get $|\{(t,x) \in Q : \alpha^{\mu_1} \neq \alpha^{\mu_2}, m_i \neq 0\}| = 0$ for $i = 1,2$. Therefore, $\alpha^{\mu_1} = \alpha^{\mu_2}$. By the uniqueness of solutions to the Fokker-Planck equation, $m_1 = m_2$. Therefore, $\mu_1 = (I,\alpha^{\mu_i}) \# m_i = \mu_2$. By uniqueness of solutions to the Hamilton Jacobi equation, $u_1 = u_2$.
\end{proof}

	\bibliographystyle{alpha}
    \bibliography{mybib}

\end{document}

%% file: wasserstein_spaces.tex

We denote by $\mathcal{P}(\mathbb{T}^d)$ the set of all (Borel) probability measures on $\bb{T}^d$.
We denote by $\s{P}_q(\bb{R}^d)$ the set of all probability measures on $\bb{R}^d$ having finite $q$th moment, i.e.~such that $\int_{\bb{R}^d} \abs{x}^q\dif m(x) < \infty$.
On $\s{P}_q(\bb{R}^d)$ or on $\s{P}(\bb{T}^d)$ we can define a metric, the $q$-Wasserstein distance, given by
\begin{equation}
    W_q(m_1,m_2) = \inf \cbr{\int \abs{x-y}^q \dif \pi(x,y) : \pi \in \Pi(m_1,m_2)}^{1/q}
\end{equation}
where $\Pi(m_1,m_2)$ denotes the set of all couplings of $m_1$ and $m_2$, i.e.~the set of all probability measures $\pi$ on $\bb{R}^d \times \bb{R}^d$ (or $\bb{T}^d \times \bb{T}^d$) whose first and second marginals are $m_1$ and $m_2$, respectively.
Equivalently,
\begin{equation}
    W_q(m_1,m_2) = \inf \cbr{\bb{E}\abs{X-Y}^q : X \sim m_1, \ Y \sim m_2}^{1/q}
\end{equation}
where $X \sim m$ means $X$ is  random variable whose law is $m$. 
A special case is the 1-Wasserstein distance, which is also given by
\begin{equation}
	W_1(m_1,m_2) = \sup\cbr{\int \phi(x)\dif (m_1 - m_2)(x) : \enVert{\nabla \phi}_\infty \leq 1}.
\end{equation}